\documentclass[12pt]{amsart}
\usepackage{a4wide,enumerate,xcolor}
\usepackage{amsmath,graphicx,comment,enumerate}
\usepackage{mathtools} 
\allowdisplaybreaks

\usepackage{enumitem}
\setlist[itemize]{label={$\bullet$}, leftmargin=12pt, itemsep=3pt}

\let\pa\partial
\let\na\nabla
\let\eps\varepsilon
\newcommand{\N}{{\mathbb N}}
\newcommand{\R}{{\mathbb R}}
\newcommand{\diver}{\operatorname{div}}

\newtheorem{theorem}{Theorem}
\newtheorem{lemma}[theorem]{Lemma}
\newtheorem{proposition}[theorem]{Proposition}
\newtheorem{remark}[theorem]{Remark}

\begin{document}

\title[A degenerate Keller--Segel system with volume filling]{Well-posedness and asymptotic limits for a degenerate Keller--Segel system with volume filling}

\author[N. Geltner]{Noah Geltner}
\address{Institute of Analysis and Scientific Computing, TU Wien, Wiedner Hauptstra\ss e 8--10, 1040 Wien, Austria}
\email{noah.geltner@tuwien.ac.at} 

\author[A. J\"ungel]{Ansgar J\"ungel}
\address{Institute of Analysis and Scientific Computing, TU Wien, Wiedner Hauptstra\ss e 8--10, 1040 Wien, Austria}
\email{juengel@tuwien.ac.at} 

\author[M. Zhang]{Mingyue Zhang}
\address{Institute of Analysis and Scientific Computing, TU Wien, Wiedner Hauptstra\ss e 8--10, 1040 Wien, Austria}
\email{mingyue.zhang@tuwien.ac.at} 

\date{\today}

\thanks{The authors acknowledge partial support from the Austrian Science Fund (FWF), grant 10.55776/PAT2687825, and from the Austrian Federal Ministry for Women, Science and Research and implemented by \"OAD, project MultHeFlo. This work has received funding from the European Research Council (ERC) under the European Union's Horizon 2020 research and innovation programme, ERC Advanced Grant NEUROMORPH, no.~101018153. For open-access purposes, the authors have applied a CC BY public copyright license to any author-accepted manuscript version arising from this submission.} 

\begin{abstract}
A class of parabolic--parabolic Keller--Segel systems with degenerate diffusion and volume filling is studied in a bounded domain subject to no-flux boundary conditions. The equations are derived from a multiphase fluid model. The interplay between nonlinear diffusion and density saturation leads to a rich variety of behaviors across different parameter regimes. We establish the existence of global weak solutions, a weak--strong uniqueness result, the exponential convergence to the homogeneous steady state, pattern formation in one spatial dimension, as well as the parabolic–elliptic and vanishing diffusion limits. The analysis relies on a priori estimates derived from suitable entropy functionals. Pattern formation is demonstrated by reducing the system to a first-order equation and conducting a detailed analysis of the resulting nonlinearity. Numerical simulations from a one-dimensional finite-volume scheme illustrate the asymptotic regimes.
\end{abstract}

\keywords{Keller--Segel equations, existence of weak solutions, weak--strong uniqueness, steady states, pattern formation, parabolic--elliptic limit, vanishing diffusion limit.}  
 
\subjclass[2000]{35K51, 35K65, 35B36, 35Q92, 92C17.}

\maketitle


\section{Introduction}

The Keller--Segel equations provide a fundamental mathematical framework for describing how populations of cells sense and collectively migrate along chemical signals \cite{KeSe70}. While the classical Keller--Segel model assumes linear diffusion and constant sensitivity, generalized models account for finite-speed propagation or the prevention of overcrowding. Mathematically, this can be modeled by degenerate diffusion \cite{BDD06,PeZh24} or a modification of the chemotactic sensitivity which diminishes as the cell density approaches a maximal packing threshold \cite{CCWWZ20,PaHi02}. The resulting system, derived from a multiphase fluid model and studied here for the first time, combines degenerate diffusion with nonlinear, density-limited chemotactic drift, yielding a rich mathematical structure: It prevents unphysical blowup present in classical Keller–Segel models, while still capturing pattern formation, aggregation, and saturation phenomena observed in real biological systems. Moreover, depending on the nonlinear diffusion, we distinguish parameter regimes in which pattern formation either emerges or is inhibited.

The Keller--Segel equations for the cell density $\rho(t,x)$ and the concentration of the chemoattractant $c(t,x)$ are given by
\begin{align}
  \pa_t\rho &= \diver\big((1-\rho)\rho^{m-1}\na\rho 
  - \chi\rho(1-\rho)\na c\big), \label{1.rho} \\
  \tau\pa_t c &= \eta\Delta c - c + \rho\quad\mbox{in }\Omega,\ t>0,
  \label{1.c}
\end{align}
supplemented with the initial and no-flux boundary conditions
\begin{equation}\label{1.bic}
\begin{aligned}
  & \rho(0)=\rho^0, \quad c(0)=c^0, \quad\mbox{in }\Omega, \\
  & (1-\rho)\rho^{m-1}\na\rho\cdot\nu 
  = \na c\cdot\nu = 0\quad\mbox{on }\pa\Omega,\ t>0,
\end{aligned}
\end{equation}
where $0\le \rho^0,c^0\le 1$ in $\Omega$. The biological parameters are the degeneracy exponent $m>1$, the chemotactic sensitivity $\chi$, the characteristic time $\tau$, and the chemical diffusion rate $\eta>0$. 

Both diffusion and drift in \eqref{1.rho} vanish as the cell density approaches a maximal packing threshold. The exponent $m>1$, which expresses slow diffusion, plays a crucial role in pattern formation: For instance, steady states (more precisely: equilibrium solutions) are positive for $m=1$, whereas for $m>1$, they may vanish on subsets of the domain. 

Our model can be derived formally from mass and momentum balance equations in the multiphase approach by assuming that the total interaction force is given by interphase, chemotactic, and drag forces; see Appendix \ref{app}. Equations \eqref{1.rho}--\eqref{1.c} can also be interpreted as the formal gradient flow
\begin{equation*}
\begin{aligned}
  & \pa_t\rho = \diver\bigg(M(\rho)\na\frac{\delta E}{\delta\rho}\bigg),
  \quad \chi\tau\pa_t c = -\frac{\delta E}{\delta c}, \\
  & \mbox{where}\quad E(\rho,c) = \int_\Omega
  \bigg(\frac{\rho^m}{m(m-1)} + \frac{\chi}{2}(\eta|\na c|^2 
  + c^2) - \chi\rho c\bigg)dx
\end{aligned}
\end{equation*}
is the energy and the mobility reads as $M(\rho)=\rho(1-\rho)$, used in \cite{BDD06}. Such a relation also holds in the case $m=1$, with $\rho^m/(m(m-1))$ replaced by $\rho(\log\rho-1)$.

The mathematical modeling of chemotaxis phenomena dates back to the pioneering works of Patlak in the 1950s \cite{Pat53} and Keller and Segel in the 1970s \cite{KeSe70}. Since then, the body of mathematical literature on the Keller--Segel equations has grown extensively. We highlight only some contributions concerning degenerate and volume-filling models. For notational convenience, we write the general Keller--Segel equations as
\begin{align}\label{1.general}
  \pa_t\rho = \diver\big(a(\rho)\na\rho - b(\rho)\na c\big),
  \quad \tau\pa_t c = \eta\Delta c - c + \rho,
\end{align}
where $a$ and $b$ are nonlinear functions. 

Degenerate Keller--Segel equations with $a(\rho)=\rho^{m-1}$ ($m>1$) and $b(\rho)=\chi\rho$ exhibit a finite speed of propagation of the cell movement, which seems to be a realistic cell behavior. The existence of global weak solutions for $m>3-4/d$ was proved in \cite{KoSz08} and later improved to the optimal range $m>2-2/d$ \cite{TaWi12}. The convergence to possibly nonconstant steady states  for $m\ge 2$ and to the constant steady state for $m=2$ was proved in \cite{Jia18}. A complete phase transition diagram for the one-dimensional model with $m=2$ was given in \cite{CCWWZ20}. 

Hillen and Painter \cite{PaHi02} derived a volume-filling Keller--Segel system from a lattice model with the functions $a(\rho)=q(\rho)-\rho q'(\rho)$ and $b(\rho)=\rho q(\rho)$, where $q$ is nonnegative and satisfies $q(1)=0$. The simplest choice $q(\rho)=1-\rho$ leads to $a(\rho)=1$ (analyzed in \cite{HiPa01}); thus, this class of models is different from our system. Conditions on $a(\rho)$ and $b(\rho)$ such that the solutions do not blow up were given in \cite{BDD06,CaCa06} for the elliptic case $\tau=0$ and in \cite{DiRo08} for the parabolic case $\tau>0$. In \cite{WaHi07}, pattern formation of the parabolic model $\tau>0$ was investigated. For a comprehensive review of volume-filling chemotaxis models, we refer the reader to \cite{Wrz10}.

Also combined degenerate and volume-filling Keller--Segel equations have been studied in the literature. The model of Hillen and Painter with the choice $q(\rho) = (1-\rho)^r$ with $r>1$ was investigated in \cite{WWW12}, and general functions $q$ have been considered in \cite{LaWr05}. Bendahmane and co-workers \cite{BKU07} proved the existence of H\"older continuous solutions to \eqref{1.general} with $a(0)=a(1)=0$ and $b(1)=0$. Although these properties are satisfied in our model, the other assumptions imposed in \cite{BKU07} do not hold in our situation. Solutions to \eqref{1.general} blow up in finite time if the nonlinear chemosensitivity is strong enough and an adequate relation between nonlinear diffusion and chemosensitivity holds \cite{CiSt12}. 
Pattern formation for $a(\rho)=\rho^{m-1}$ and $b(\rho)=\chi\rho(1-\rho)$ was shown in \cite{PeZh24} when $m>2$, while the steady state for $m\le 2$ is shown to be unique and constant. Notice that in our model, the function $a(\rho)=(1-\rho)\rho^{m-1}$ features a stronger nonlinearity, which complicates the analysis. Finally, we mention the reviews \cite{ArTy21,BBTW15}, which provide an overview of the various variants of the Keller--Segel equations.

In this paper, we prove the existence of global weak solutions to  \eqref{1.rho}--\eqref{1.bic}, investigate the weak--strong uniqueness, study pattern formation for $m>2$, study the long-time behavior for $m\le 2$, as well as perform the parabolic--elliptic limit $\tau\to 0$ and the vanishing diffusion limit $\eta\to 0$. In the following, we make precise our main results.

Let $\Omega\subset\R^d$ ($d\ge 1$) be a bounded domain with boundary $\pa\Omega\in C^{1,1}$ and let $\chi$, $\tau$, $\eta>0$. The initial data satisfies $\rho^0\in L^m(\Omega)$, $c^0\in H^1(\Omega)$, and $0\le\rho^0$, $c^0\le 1$ in $\Omega$. We set $\Omega_T=(0,T)\times\Omega$ for $T>0$. Our first result is the existence of global weak solutions.

\begin{theorem}[Existence of weak solutions]\label{thm.ex}
Let $m\ge 1$ and set $f(\rho)= \rho^{m}/m-\rho^{m+1}/(m+1)$. Then there exists a weak solution $(\rho,c)$ to \eqref{1.rho}--\eqref{1.bic} on $[0,\infty)$ satisfying $0\le\rho,\,c\le 1$ a.e.\ in $[0,\infty)\times\Omega$,
\begin{align*}
  & f(\rho)\in L^2(0,T;H^1(\Omega)), \quad
  c\in L^\infty(0,T;H^1(\Omega)), \\ 
  & \pa_t\rho\in L^2(0,T;H^1(\Omega)'), \quad 
  \pa_t c,\,\Delta c\in L^2(\Omega_T)
\end{align*}
for any $T>0$, and $(\rho,c)$ satisfies \eqref{1.rho}--\eqref{1.c} in the weak sense. 
\end{theorem}

To prove this theorem, we regularize equation \eqref{1.rho} by adding the diffusion $\eps\Delta\rho$ with $\eps>0$ to \eqref{1.rho} and truncating the nonlinearities. The regularized problem is solved by a fixed-point argument. The approximate entropy inequality yields estimates uniform in $\eps$, which allow for the limit $\eps\to 0$ by applying the Aubin--Lions lemma in the version of \cite{Mou16} to the (monotone) function $f(\rho)$. 

The proof of uniqueness of weak solutions to degenerate Keller--Segel systems is very delicate, and we are aware of results mainly for $a(\rho)=\rho^m$ and $b(\rho)=\chi\rho$ only \cite{KaSu14,KoSz08}. A nonlocal sensitivity was considered in \cite{HWS16}. The proof in \cite{DGJ01} requires that $b'$ vanishes at $\rho=0$, which is not the case in our model. Therefore, we prove the (weaker) weak--strong uniqueness property. The aim is to show that a weak solution and a strong solution (if it exists) emanating from the same initial data coincide. Weak--strong uniqueness means that classical solutions are stable within the class of weak solutions. 

\begin{theorem}[Weak--strong uniqueness]\label{thm.wsu}
Let $1<m\le 2$, let $(\rho,c)$ and $(\bar\rho,\bar{c})$ be two weak solutions to \eqref{1.rho}--\eqref{1.bic} with the same initial data. Assume that there exists $\kappa\in(0,1)$ such that $\kappa\le\bar\rho\le 1-\kappa$ in $\Omega_T$ and that
$\na\bar{\rho}$, $\na\bar{c}\in L^\infty(\Omega_T)$. Furthermore, we suppose that $(\bar\rho,\bar{c})$ satisfies for all $\lambda>0$ the entropy equality
\begin{align*}
  \int_\Omega&\bigg(\bar\rho(\log\bar\rho-1)
  + (1-\bar\rho)(\log(1-\bar\rho)-1) 
  + \frac{\lambda}{2}\tau c^2\bigg)dx\bigg|_0^t \\
  &+ \int_0^t\int_\Omega\big(\bar{\rho}^{m-2}|\na\bar\rho|^2
  + \lambda(\eta|\na\bar{c}|^2 + \bar{c}^2)\big)dxds
  = \int_0^t\int_\Omega\big(\chi\na\bar{c}\cdot\na\bar\rho 
  + \lambda\bar\rho\bar{c}\big)dxds. \nonumber 
\end{align*}
Then $\rho=\bar\rho$ and $c=\bar{c}$ a.e.\ in $\Omega_T$.
\end{theorem}

The proof of the weak--strong uniqueness property is based on the relative entropy
\begin{align*}
  H(\rho,c|\bar\rho,\bar{c}) = \int_\Omega\bigg(
  \rho\log\frac{\rho}{\bar\rho} + (1-\rho)\log\frac{1-\rho}{1-\bar\rho}
  + \frac{\lambda}{2}\tau(c-\bar{c})^2\bigg)dx,
\end{align*}
where $(\rho,c)$ is a weak solution and $(\bar\rho,\bar{c})$ is a strong solution (i.e.\ a weak solutions satisfying the regularity stated in Theorem \ref{thm.wsu}) with the same initial data. The aim is to show that there exists $C>0$ such that
\begin{align*}
  \frac{d}{dt}H(\rho,c|\bar\rho,\bar{c})
  \le CH(\rho,c|\bar\rho,\bar{c}) \quad\mbox{for }t>0,
\end{align*}
which implies that $H((\rho,c)(t)|(\bar\rho,\bar{c})(t))=0$ and hence $\rho(t)=\bar\rho(t)$, $c(t)=\bar{c}(t)$ for all $t>0$. The restriction on the parameter $m$ comes from the different exponents in the diffusion and drift terms of \eqref{1.rho} that need to be estimated. The condition $\kappa\le\bar\rho\le 1-\kappa$ is technical; it ensures that the relative entropy is well defined.

Next, we investigate the steady states and the long-time behavior. More precisely, we consider equilibrium solutions, i.e.\ stationary solutions with vanishing flux. We show that for $1<m\le 2$, the steady state is unique and constant; see Proposition \ref{prop.unique}. When $m>2$ and $\chi$ is in some range, we prove for the one-dimensional equations that there exist increasing steady states, thus showing pattern formation; see Proposition \ref{prop.inc}. The proof relies on a reduction of the problem to a first-order equation for $c$ and a careful analysis of the involved nonlinearities and parameter dependencies. Because of these results, the convergence to the (unique) steady state can only be expected in the case $1<m\le 2$.

\begin{theorem}[Exponential time decay]\label{thm.time}
Let $1<m\le 2$, $\chi\le 1$, $|\Omega|=1$, $M=\int_\Omega \rho_0 dx = \int_\Omega c^0 dx$, and let $(\rho,c)$ be the solution constructed in Theorem \ref{thm.ex}. Then there exist $C>0$, depending on $(\tau, \rho^0,c^0,\Omega,M)$, and $\mu>0$, depending on $(\eta, \tau, \Omega,M)$, such that for all $t>0$,
\begin{align*}
  \|\rho(t)-M\|_{L^2(\Omega)} + \|c(t)-M\|_{H^1(\Omega)}
  \le C e^{-\mu t}.
\end{align*} 
\end{theorem}

The proof is based on the relative entropy method, similar as for the weak--strong uniqueness proof, but here for the functional 
\begin{align}\label{1.relent}
  H_1(\rho,c|M,M) = \int_\Omega\bigg(\rho\log\frac{\rho}{M}
  + (1-\rho)\log\frac{1-\rho}{1-M} + \frac{\tau}{2}|\na c|^2\bigg)dx.
\end{align}
We prove that there exists $\mu>0$ such that for $t>0$,
\begin{align*}
  \frac{d}{dt}H_1(\rho,c|M,M) + \mu H_1(\rho,c|M,M) \le 0,
\end{align*}
implying that $H_1((\rho,c)(t)|M,M)\le H_1(\rho^0,c^0|M,M)e^{-\mu t}$. The result then follows from a Taylor expansion of the relative entropy and an $L^2(\Omega)$-type estimation for equation \eqref{1.c}.

We also prove the limit $\tau\to 0$, leading to the parabolic--elliptic model, and the vanishing diffusion limit $\eta\to 0$. 

\begin{theorem}[Parabolic--elliptic limit]\label{thm.D}
Let $m>1$ and let $(\rho_\tau,c_\tau)$ be a weak solution to \eqref{1.rho}--\eqref{1.bic}. Then there exists a subsequence (not relabeled) such that, as $\tau\to 0$,
\begin{align*}
  \rho_\tau\to\rho\quad\mbox{strongly in }L^p(\Omega_T)\mbox{ for all }
  p<\infty, \quad
  c_\tau\to c\quad\mbox{strongly in }L^2(0,T;H^1(\Omega)),
\end{align*}
and $(\rho,c)$ is a weak solution to 
\begin{align}\label{1.tau}
  \pa_t\rho = \diver\big((1-\rho)\rho^{m-1}\na\rho
  - \chi\rho(1-\rho)\na c\big), \quad
  0 = -\eta\Delta c - c + \rho\quad\mbox{in }\Omega,\ t>0,
\end{align}
with the initial and boundary conditions \eqref{1.bic}.
\end{theorem}

\begin{theorem}[Vanishing diffusion limit]\label{thm.eta}
Let $1<m\le 2$ and let $(\rho_\eta,c_\eta)$ be a weak solution to \eqref{1.rho}--\eqref{1.bic}. Then there exists a subsequence (not relabeled) such that, as $\eta\to 0$,
\begin{align*}
  \rho_\eta\to \rho, \quad c_\eta\to c \quad\mbox{strongly in }
  L^2(\Omega_T),
\end{align*}
and $(\rho,c)$ is a weak solution to
\begin{align}\label{1.eta}
  \pa_t\rho = \diver\big((1-\rho)\rho^{m-1}\na\rho
  - \chi\rho(1-\rho)\na c\big), \quad
  \pa_t c = - c + \rho\quad\mbox{in }\Omega,\ t>0,
\end{align}
with the initial and boundary conditions \eqref{1.bic}.
\end{theorem}

If the limit equations are uniquely solvable, we recover the convergence for the whole sequence. To prove Theorems \ref{thm.D} and \ref{thm.eta}, we derive uniform estimates from entropy functionals. For the limit $\tau\to 0$, we use
\begin{align}\label{1.Ftau}
  F_\tau(\rho,c) = \int_\Omega\big(\rho(\log\rho-1) 
  + (1-\rho)(\log(1-\rho)-1) + \chi\tau\rho c\big)dx,
\end{align}
while the proof for $\eta\to 0$ is based on 
\begin{align}\label{1.Feta}
  F_\eta(\rho,c) = \int_\Omega \bigg(\rho(\log\rho-1) 
  + (1-\rho)(\log(1-\rho)-1) + \frac{\tau}{2}|\na c|^2\bigg)dx.
\end{align}
The associated entropy inequalities yield uniform bounds that allow us to conclude compactness of $(\rho,c)$ via the Aubin--Lions lemma. While the compactness arguments are standard, the difficulty is the justification of the entropy inequality, which is done by an approximation argument to deal with the logarithmic terms.

The paper is organized as follows. The existence result is proved in Section \ref{sec.ex}, and the weak--strong uniqueness of solutions is shown in Section \ref{sec.wsu}. In Section \ref{sec.steady}, we compute the steady states for the cases $1<m\le 2$ and $m>2$. The exponential decay of the solution to the unique steady state for $1<m\le 2$ is verified in Section \ref{sec.long}. The asymptotic limits of Theorems \ref{thm.D} and \ref{thm.eta} are proved in Section \ref{sec.asym}. We present in Section \ref{sec.num} numerical simulations in one space dimension to illustrate the behavior of the solutions for various parameters $m$, $\tau$, and $\eta$. Finally, the derivation of equation \eqref{1.rho} from a multiphase fluid model is presented in Appendix \ref{app}. 


\section{Existence of weak solutions}\label{sec.ex}

In this section, we prove Theorem \ref{thm.ex}. Because of the degeneracy, we first solve an approximate equation with an additional linear diffusion term with parameter $\eps$ by using the Leray--Schauder fixed-point theorem. Then estimates uniform in $\eps$ allow us to perform the limit $\eps\to 0$.

\subsection{Solution to an approximate problem}

We consider the following approximate equations:
\begin{align}
  \pa_t \rho_\eps &= \diver\big(\eps\na \rho_\eps  
  + (1-\rho_\eps)\rho_\eps^{m-1}\na\rho_\eps
  - \chi\rho_\eps(1-\rho_\eps)\na c_\eps\big), \label{2.rho} \\
  \tau\pa_t c_\eps &= \eta\Delta c_\eps - c_\eps + \rho_\eps
  \quad\mbox{in }\Omega,\ t>0, \label{2.c}
\end{align}
together with the initial and boundary conditions \eqref{1.bic}. We apply a fixed-point argument. For this, let $\bar\rho\in L^2(\Omega_T)$ and $\sigma\in[0,1]$. There exists a unique weak solution $c\in L^2(0,T;H^{1}(\Omega))\cap H^{1}(0,T;H^1(\Omega)')$ to
\begin{align}\label{2.clin}
  \tau\pa_t c = \eta\Delta c - c + [\bar\rho]\quad\mbox{in }\Omega,\ t>0,
  \quad \na c\cdot\nu=0\quad\mbox{on }\pa\Omega,\quad 
  c(0)=c^0\quad\mbox{in }\Omega,
\end{align}
where $[z]=\max\{0,\min\{1,z\}\}$ for $z\in\R$.
Multiplying this equation by $(c-1)^+=\max\{0,c-1\}$, we find that
\begin{align*}
  \frac{\tau}{2}\frac{d}{dt}\int_\Omega|(c-1)^+|^2 dx
  + \int_\Omega\big(\eta|\na(c-1)^+|^2 + |(c-1)^+|^2\big) dx
  = \int_\Omega([\bar\rho]-1)(c-1)^+ dx
  \le 0.
\end{align*}
Since $c(0)=c^0\le 1$, this shows that $(c-1)^+=0$ and hence $c\le 1$ in $\Omega_T$. Similarly, using the test function $c^-=\min\{0,c\}$, we obtain $c\ge 0$ in $\Omega_T$. Now, there exists a unique weak solution $\rho\in L^2(0,T;H^1(\Omega))\cap H^1(0,T;H^1(\Omega)')$ to the linear problem
\begin{equation}\label{2.rholin}
\begin{aligned}
  & \pa_t\rho = \eps\Delta\rho 
  + \sigma\diver\big([1-\bar\rho][\bar\rho]^{m-1}\na\rho 
  - \chi[\bar\rho][1-\bar\rho]\na c\big)\mbox{in }\Omega,\ t>0, \\
  & [1-\bar\rho][\bar\rho]^{m-1}\na\rho\cdot\nu = 0
  \quad\mbox{on }\pa\Omega, \quad
  \rho(0)=\sigma\rho_0\quad\mbox{in }\Omega.
\end{aligned}
\end{equation}

This defines the fixed-point operator $S:L^2(\Omega_T)\times[0,1]\to L^2(\Omega_T)$, $S(\bar\rho,\sigma)=\rho$. It holds that $S(\bar\rho,0)=0$. The continuity of $S$ follows from standard arguments. We deduce from the compactness of the embedding $L^2(0,T;H^1(\Omega))\cap H^1(0,T;H^1(\Omega)')\hookrightarrow C^0([0,T];L^2(\Omega))$ that $S$ is  also compact.

It remains to derive a uniform bound for all fixed points of $S(\cdot,\sigma)$ for $\sigma>0$. Let $\rho$ be such a fixed point. First, with the test function $c$ in \eqref{2.clin} with $\bar\rho=\rho$,
\begin{align}\label{2.cc}
  \frac{\tau}{2}\frac{d}{dt}\int_\Omega c^2 dx
  + \int_\Omega(\eta|\na c|^2 + c^2)dx
  = \int_\Omega[\rho]cdx \le \mbox{meas}(\Omega),
\end{align}
which yields a uniform bound for $c$ in $L^2(0,T;H^1(\Omega))$. Then we use the test function $\rho$ in the weak formulation to \eqref{2.rholin} with $\bar\rho=\rho$:
\begin{align*}
  \frac12\frac{d}{dt}\int_\Omega\rho^2 dx
  + \int_\Omega(\eps + \sigma[1-\rho][\rho]^{m-1})|\na\rho|^2 dx
  &= \chi\sigma\int_\Omega[\rho][1-\rho]\na c\cdot\na\rho dx \\
  &\le \frac{\eps}{2}\int_\Omega|\na\rho|^2 dx
  + \frac{\chi^2}{2\eps}\int_\Omega|\na c|^2 dx.
\end{align*}
We conclude a uniform bound for $\rho$ in $L^\infty(0,T;L^2(\Omega))\cap L^2(0,T;H^1(\Omega))$ (not uniform in $\eps$). This provides the required bound for $\rho$ in $L^2(\Omega_T)$, and we infer the existence of a fixed point $\rho_\eps$ of $S(\cdot,1)$, which is a weak solution to \eqref{1.bic}, \eqref{2.rholin} with $\bar\rho=\rho_\eps$ and $\sigma=1$. Accordingly, let $c_\eps$ be the solution to \eqref{1.bic}, \eqref{2.clin}. 

We claim that $0\le\rho_\varepsilon\le 1$ in $\Omega_T$. Indeed, we use the test function $(\rho_\eps-1)^+$ in the weak formulation to \eqref{2.rholin} with $\bar\rho=\rho_\eps$ and $\sigma=1$:
\begin{align*}
  \frac12\frac{d}{dt}\int_\Omega|(\rho_\eps-1)^+|^2 dx
  &+ \eps\int_\Omega|\na(\rho_\eps)^+|^2 dx
  = -\int_\Omega[1-\rho_\eps][\rho_\eps]^{m-1}
  \mathrm{1}_{\{\rho_\eps>1\}}|\na\rho_\eps|^2 dx \\
  &+ \chi\int_\Omega[\rho_\eps][1-\rho_\eps]\mathrm{1}_{\{\rho_\eps>1\}}
  \na c_\eps\cdot\na\rho_\eps dx = 0,
\end{align*}
since $[1-\rho_\eps]\mathrm{1}_{\{\rho_\eps>1\}}=0$. Hence $\rho_\eps\le 1$ in $\Omega_T$. Similarly, we show that $\rho_\eps\ge 0$ in $\Omega_T$. Thus, $(\rho_\eps,c_\eps)$ solves equations \eqref{2.rho}--\eqref{2.c}.

\subsection{Uniform bounds}

We derive bounds for $\rho_\eps$ and $c_\eps$ uniform in $\eps$. We know from \eqref{2.cc} that $(c_\eps)$ is bounded in $L^2(0,T;H^1(\Omega))\cap H^1(0,T;H^1(\Omega)')$. 
Now, we use the function $\eps\rho_\eps$ in the weak formulation of \eqref{2.rho}:
\begin{align*}
  \frac{\eps}{2}\frac{d}{dt}&\int_\Omega \rho_\eps^2 dx
  + \eps^2\int_\Omega|\na\rho_\eps|^2 dx
  + \eps\int_\Omega\rho_\eps(1-\rho_\eps)\rho_\eps^{m-1}
  |\na\rho_\eps|^2 dx \\
  &= \chi\eps\int_\Omega\rho_\eps(1-\rho_\eps)\na c_\eps
  \cdot\na\rho_\eps dx 
  \le \frac{\eps^2}{2}\int_\Omega|\na\rho_\eps|^2 dx
  + \frac{\chi^2}{2}\int_\Omega|\na c_\eps|^2 dx.
\end{align*}
We infer that, for a constant $C>0$ independent of $\eps$,
\begin{align}\label{2.epsrho}
  \eps\|\na\rho_\eps\|_{L^2(\Omega_T)} \le C.
\end{align} 
To derive a uniform gradient bound involving $\rho_\eps$, we introduce the function
\begin{align*}
  \phi_\eps(z) = \int_{1/2}^{z}
  \bigg(\frac{\eps}{s(1-s)} + s^{m-2}\bigg)ds
  = \eps\log\frac{z}{1-z} + \frac{z^{m-1}}{m-1} 
  - \frac{1}{m-1}\bigg(\frac12\bigg)^{m-1}
\end{align*}
and set $\Phi_\eps(z)=\int_0^z\phi_\eps(s)ds$ for $0<z<1$. Equation \eqref{2.rho} can be written equivalently as
\begin{align}\label{2.rho2}
  \pa_t\rho_\eps = \diver\big(\rho_\eps(1-\rho_\eps)
  \na(\phi_\eps(\rho_\eps)
  -\chi c_\eps)\big)\quad\mbox{in }\Omega,\ t>0.
\end{align}
Since $\phi(\rho_\eps)$ is not an admissible test function, we need to regularize:
\begin{align*}
  \phi_\eps^\delta(z) = \int_{1/2}^{z}
  \bigg(\frac{\eps}{(s+\delta)(1-s+\delta)}
  + (s+\delta)^{m-2}\bigg)ds, \quad
  \Phi_\eps^\delta(z) = \int_0^{z}\phi_\eps^\delta(s)ds,
\end{align*}
where now, we can choose $0\le z\le 1$. Then $\phi_\eps^\delta(\rho_\eps)$ is an admissible test function in \eqref{2.rho}:
\begin{align}\label{2.aux}
  \int_\Omega\Phi_\eps^\delta(\rho_\eps)dx\Big|_{0}^t
  &= -\int_0^t\int_\Omega\rho_\eps(1-\rho_\eps)
  \phi'_\eps(\rho_\eps)(\phi_\eps^\delta)'(\rho_\eps)
  |\na\rho_\eps|^2 dxds \\
  &\phantom{xx}+ \chi\int_0^t\int_\Omega\rho_\eps(1-\rho_\eps)
  (\phi_\eps^\delta)'(\rho_\eps)\na c_\eps\cdot\na\rho_\eps dxds \nonumber \\
  &= -\int_0^t\int_\Omega|\na A_\delta(\rho_\eps)|^2 dxds 
  + \chi\int_0^t\int_\Omega\na c_\eps\cdot\na B_\delta(\rho_\eps) 
  dxds, \nonumber 
\end{align}
where $A'_\delta(z) = \sqrt{z(1-z)\phi_\eps'(z)(\phi_\eps^\delta)'(z)}$ and $B'_\delta(z) = z(1-z)(\phi_\eps^\delta)'(z)$. Since $B'_\delta$ is bounded uniformly in $\delta$, $B_\delta(\rho_\eps)$ is bounded in $L^2(0,T;H^1(\Omega))$ uniformly in $\delta$. Hence, \eqref{2.aux} shows that $A_\delta(\rho_\eps)$ is bounded in $L^2(0,T;H^1(\Omega))$ as well. By compactness, we can extract a subsequence (not relabeled) such that, $A_\delta(\rho_\eps)$ and $B_\delta(\rho_\eps)$ converge weakly in $L^2(\Omega_T)$ as $\delta\to 0$. By dominated convergence, using $A_\delta(z)\to A_0(z)$, $B_\delta(z)\to B_0(z)$ a.e. for $z\in(0,1)$, we can identify the limits. We conclude that
\begin{align*}
  A_\delta(\rho_\eps)\to A_0(\rho_\eps), \quad
  B_\delta(\rho_\eps)\to B_0(\rho_\eps) \quad
  \mbox{strongly in }L^2(\Omega_T).
\end{align*}
Similarly, by dominated convergence, $\Phi_\eps^\delta(\rho_\eps)\to \Phi_\eps(\rho_\eps)$ in $L^1(\Omega)$.
Thus, we can pass to the limit $\delta\to 0$ in \eqref{2.aux} and, together with the weak lower semicontinuity of the quadratic terms, we obtain
\begin{align}\label{2.aux2}
  \int_\Omega\Phi_\eps(\rho_\eps)dx\Big|_0^t 
  + \int_0^t\int_\Omega|\na A_0(\rho_\eps)|^2 dxds
  &\le \chi\int_0^t\int_\Omega\na c_\eps\cdot\na B_0(\rho_\eps)dxds \\
  &= \chi\int_0^t\int_\Omega\rho_\eps(1-\rho_\eps)
  \na c_\eps\cdot\na\phi(\rho_\eps) dxds. \nonumber 
\end{align}
Using the test function $-\chi c_\eps$ in the weak formulation of \eqref{2.rho2}, we have
\begin{align*}
  -\chi\int_0^t\int_\Omega\rho_\eps(1-\rho_\eps)\na(\phi_\eps(\rho_\eps)
  -\chi c_\eps)\cdot\na c_\eps dxds
  = \chi\int_0^t\langle \pa_t\rho_\eps,c_\eps\rangle ds.
\end{align*}
Then we add this equation to \eqref{2.aux2}, and observing that $|\na A_0(\rho_\eps)|^2 = \rho_\eps(1-\rho_\eps)|\na\phi_\eps(\rho_\eps)|^2$, we arrive at
\begin{align*}
  \int_\Omega\Phi(\rho_\eps)dx\Big|_0^t
  + \int_0^t\int_\Omega\rho_\eps(1-\rho_\eps)
  |\na(\phi_\eps(\rho_\eps)-\chi c_\eps)|^2 dxds
  \le \chi\int_0^t\langle\pa_t\rho_\eps,c_\eps\rangle ds.
\end{align*}
We rewrite the right-hand side by replacing $\rho_\eps$ by $\tau\pa_t c_\eps - \eta\Delta c_\eps + c_\eps$:
\begin{align*}
  \langle\pa_t\rho_\eps,c_\eps\rangle
  &= \frac{d}{dt}\int_\Omega\rho_\eps c_\eps dx
  - \int_\Omega\rho_\eps\pa_t c_\eps dx \\
  &= \frac{d}{dt}\int_\Omega\rho_\eps c_\eps dx
  - \int_\Omega(\tau\pa_t c_\eps - \eta\Delta c_\eps + c_\eps)
  \pa_t c_\eps dx \\
  &= -\tau\int_\Omega|\pa_t c_\eps|^2 dx
  - \frac12\frac{d}{dt}\int_\Omega(\eta|\na c_\eps|^2 + c_\eps^2)dx.
\end{align*}
An integration over time leads to
\begin{align}\label{2.phic}
  \int_\Omega\bigg(\Phi(\rho_\eps) &+ \frac{\chi}{2}
  (\eta|\na c_\eps|^2 + c_\eps^2) - \chi\rho_\eps c_\eps\bigg)dx\Big|_0^t 
  + \chi\tau\int_0^t\int_\Omega|\pa_s c_\eps|^2dxds \\
  &+ \int_0^t\int_\Omega\rho_\eps(1-\rho_\eps)
  |\na(\phi_\eps(\rho_\eps)-\chi c_\eps)|^2 dxds \le 0. \nonumber 
\end{align}
Because of $\rho_\eps c_\eps\le 1$, this gives the estimate
\begin{align}\label{2.estphi}
  \int_0^t\int_\Omega\rho_\eps(1-\rho_\eps)
  |\na(\phi_\eps(\rho_\eps)|^2 dxds
  &\le 2\int_0^t\int_\Omega \rho_\eps(1-\rho_\eps)
  |\na(\phi_\eps(\rho_\eps)-\chi c_\eps)|^2 dxds \\
  &\phantom{xx}+ 2\chi^2\int_0^t\int_\Omega\rho_\eps(1-\rho_\eps)
  |\na c_\eps|^2 dxds \le C, \nonumber 
\end{align}
where we used bound \eqref{2.cc} in the last step. 

\subsection{Limit $\eps\to 0$}

Let $f(z)=z^m/m - z^{m+1}/(m+1)$ for $z\ge 0$. Then $(f(\rho_\eps))_\eps$ is bounded in $L^\infty(\Omega_T)$ and
\begin{align*}
  \|\na f(\rho_\eps)\|_{L^2(\Omega_T)}
  &= \|(1-\rho_\eps)\rho_\eps^{m-1}\na\rho_\eps\|_{L^2(\Omega_T)} \\
  &\le\bigg\|\rho_\eps(1-\rho_\eps)\bigg(
  \frac{\eps}{\rho_\eps(1-\rho_\eps)} + \rho_\eps^{m-2}\bigg)
  \na\rho_\eps\bigg\|_{L^2(\Omega_T)} 
  + \eps\|\na\rho_\eps\|_{L^2(\Omega_T)} \\
  &= \|\rho_\eps(1-\rho_\eps)\na\phi_\eps(\rho_\eps)\|_{L^2(\Omega_T)}
  + \eps\|\na\rho_\eps\|_{L^2(\Omega_T)} \le C,
\end{align*}
where the last step follows from \eqref{2.epsrho} and \eqref{2.estphi}. Furthermore, by \eqref{2.rho2} and \eqref{2.phic},
\begin{align*}
  \|\pa_t\rho_\eps\|_{L^2(0,T;H^1(\Omega)')}
  \le \|\rho_\eps(1-\rho_\eps)\na(\phi_\eps(\rho_\eps)
  - \chi c_\eps)\|_{L^2(\Omega_T)} \le C.
\end{align*}
The function $f$ and the family $(\rho_\eps)$ satisfy the conditions of the Aubin--Lions lemma of \cite[Theorem 1]{Mou16}, which allows us to extract a subsequence that is not relabeled such that
\begin{align*}
  \rho_\eps\to \rho \quad\mbox{strongly in }L^2(\Omega_T)
  \quad\mbox{as }\eps\to 0.
\end{align*}
Thanks to the uniform $L^\infty(\Omega_T)$ bound of $\rho_\eps$, this convergence holds in $L^p(\Omega_T)$ for any $p<\infty$. In particular, up to a subsequence, $\rho_\eps\to\rho$ a.e.\ in $\Omega_T$. Thus, $f(\rho_\eps)\to f(\rho)$ a.e.\ in $\Omega_T$ and
\begin{align*}
  \na f(\rho_\eps)\rightharpoonup\na f(\rho)
  \quad\mbox{weakly in }L^2(\Omega_T).
\end{align*}
Moreover, $\pa_t\rho_\eps\rightharpoonup\pa_t\rho$ weakly in $L^2(0,T;H^1(\Omega)')$. The strong convergence of $(\rho_\eps)$ implies that $\eps\na\rho_\eps\to 0$ in the sense of distributions. Consequently,
\begin{align*}
  \rho_\eps(1-\rho_\eps)\na\phi_\eps(\rho_\eps)
  = \eps\na\rho_\eps + \na f(\rho_\eps)\to \na f(\rho)
  \quad\mbox{in the sense of distributions}.
\end{align*}
Thanks to the bounds for $(c_\eps)$ from \eqref{2.cc}, we can apply the Aubin--Lions lemma to infer that (up to a subsequence)
\begin{align*}
  c_\eps\to c\quad\mbox{strongly in }L^2(\Omega_T).
\end{align*}
Moreover, $\na c_\eps\rightharpoonup\na c$ weakly in $L^2(\Omega_T)$ and $\pa_t c_\eps\rightharpoonup\pa_t c$ weakly in $L^2(0,T;H^1(\Omega)')$. Hence,
\begin{align*}
  \rho_\eps(1-\rho_\eps)\na c_\eps \rightharpoonup
  \rho(1-\rho)\na c\quad\mbox{weakly in }L^2(\Omega_T).
\end{align*}
Thus, we can pass to the limit in the weak formulation of \eqref{2.rho}--\eqref{2.c} to find that
\begin{align*}
  \int_0^T\langle\pa_t\rho,\psi\rangle dt
  + \int_0^T\int_\Omega\rho(1-\rho)\na(\phi(\rho)-\chi c)
  \cdot\na\psi dxdt &= 0, \\
  \tau\int_0^T\langle\pa_t c,\psi\rangle dt
  + \int_0^T\int_\Omega(\eta\na c\cdot\na\psi + c\psi - \rho\psi)dxdt
  &= 0
\end{align*}
for all $\psi\in C_0^\infty(\Omega_T)$. By density, the weak formulations also hold for all $\psi\in L^2(0,T;H^1(\Omega))$. 
Since $c^0\in H^1(\Omega)$, elliptic regularity implies that $\Delta c$, $\pa_t c\in L^2(\Omega_T)$. 


\section{Weak--strong uniqueness of solutions}\label{sec.wsu}

In this section, we prove Theorem \ref{thm.wsu}. Introducing the entropy density 
\begin{align*}
  h(\rho,c) = \rho(\log\rho-1) + (1-\rho)(\log(1-\rho)-1) 
  + \frac{\lambda}{2} c^2,
\end{align*}
where $\lambda>0$ will be determined later, the relative entropy is computed according to 
\begin{align}\label{3.relent}
  H(\rho,\bar\rho|c,\bar{c}) &= \int_\Omega\big(h(\rho,c)
  - h(\bar\rho,\bar{c}) - h'(\bar\rho,\bar{c})\cdot
  (\rho-\bar\rho,c-\bar{c})^T\big)dx \\
  &= \int_\Omega\bigg(\rho\log\frac{\rho}{\bar\rho}
  + (1-\rho)\log\frac{1-\rho}{1-\bar\rho}
  + \frac{\lambda}{2}(c-\bar{c})^2\bigg)dx. \nonumber 
\end{align}
We first show the associated entropy inequality for weak solutions and then prove the weak--strong uniqueness property. 

\begin{lemma}[Entropy inequality]
The weak solution constructed in Theorem \ref{thm.ex} satisfies for a.e.\ $0<t<T$,
\begin{align}\label{3.entineq}
  \int_\Omega h(\rho,c)dx\Big|_0^t
  &+ \int_0^t\int_\Omega\bigg(\rho^{m-2}|\na\rho|^2
  + \frac{\lambda}{\tau}(\eta|\na c|^2 + c^2)\bigg)dxds \\
  &\le \int_0^t\int_\Omega\bigg(\chi\na c\cdot\na\rho 
  + \frac{\lambda}{\tau}\rho c\bigg)dxds. \nonumber 
\end{align}
\end{lemma}

\begin{proof}
We show below that
\begin{align}\label{3.claim}
  \int_\Omega\big(\rho(\log\rho-1) &+ (1-\rho)(\log(1-\rho)-1)\big)
  dx\Big|_0^t + \int_0^t\int_\Omega \rho^{m-2}|\na\rho|^2 dxds \\
  &\le \chi\int_0^t\int_\Omega\na\rho\cdot\na c dxds. \nonumber 
\end{align}
Adding this inequality and the identity
\begin{align*}
  \frac{\lambda}{2}\int_\Omega c^2 dx\Big|_0^t
  &= \frac{\lambda}{2}\int_0^t\frac{d}{dt}\int_\Omega c^2 dxds
  = \lambda\int_0^t\int_\Omega c\pa_t c dxds \\
  &= -\frac{\lambda}{\tau}\int_0^t\int_\Omega(\eta|\na c|^2 
  + c^2 - c\rho)dxds
\end{align*}
then proves the lemma. 

We proceed with the proof of \eqref{3.claim}. Let $(\rho_\eps,c_\eps)$ be a solution to the approximate problem \eqref{1.bic}, \eqref{2.rho}--\eqref{2.c}. For some given $0<\delta<1$, we set 
\begin{align*}
  F_\delta(\rho_\eps) = (\rho_\eps+\delta)
  \big(\log(\rho_\eps+\delta)-1\big)
  + (1-\rho_\eps+\delta)\big(\log(1-\rho_\eps+\delta)-1\big).
\end{align*}
We use the test function $\log(\rho_\eps+\delta) - \log(1-\rho_\eps+\delta)$ in the weak formulation of \eqref{2.rho}:
\begin{align*}
  \int_\Omega F_\delta(\rho_\eps)dx\Big|_0^t
  &= -\int_0^t\int_\Omega\big(\eps+(1-\rho_\eps)\rho_\eps^{m-1}\big)
  \frac{1+2\delta}{(\rho_\eps+\delta)(1-\rho_\eps+\delta)}
  |\na\rho_\eps|^2 dxds \\
  &\phantom{xx}+ \chi\int_0^t\int_\Omega\rho_\eps(1-\rho_\eps)
  \frac{1+2\delta}{(\rho_\eps+\delta)(1-\rho_\eps+\delta)}
  \na c_\eps\cdot\na\rho_\eps dxds.
\end{align*}
We neglect the $\eps$-part in the first integral on the right-hand side and rewrite the other expressions by defining
\begin{align*}
  C_\delta(\rho_\eps) &= \int_0^{\rho_\eps}
  z(1-z)\frac{1+2\delta}{(z+\delta)(1-z+\delta)}dz, \\
  D_\delta(\rho_\eps) &= \int_0^{\rho_\eps}
  \bigg((1-z)z^{m-1}\frac{1+2\delta}{(z+\delta)(1-z+\delta)}
  \bigg)^{1/2}dz.
\end{align*}
This yields
\begin{align}\label{3.F}
  \int_\Omega F_\delta(\rho_\eps)dx\Big|_0^t
  \le -\int_0^t\int_\Omega|\na D_\delta(\rho_\eps)|^2 dxds
  + \chi\int_0^t\int_\Omega\na c_\eps\cdot\na C_\delta(\rho_\eps)
  dxds. 
\end{align}

It follows from $(D_\delta'(z))^2 = C'_\delta(z)z^{m-2}$ for $0\le z\le 1$ that $|\na D_\delta(\rho_\eps)|^2 = C'_\delta(\rho_\eps)\rho_\eps^{m-2}|\na\rho_\eps|^2$. We conclude from the condition $1\le m\le 2$ that 
\begin{align*}
  C'_\delta(\rho_\eps)\rho_\eps^{2-m}
  = (1+2\delta)\frac{\rho_\eps}{\rho_\eps+\delta}
  \frac{1-\rho_\eps}{1-\rho_\eps+\delta}\rho_\eps^{2-m}
  \le 3
\end{align*}
and consequently
\begin{align*}
  \chi\na c_\eps\cdot\na C_\delta(\rho_\eps)
  \le \frac12|\na D_\delta(\rho_\eps)|^2
  + \frac{\chi^2}{2}C'_\delta(\rho_\eps)\rho_\eps^{2-m}|\na c_\eps|^2
  \le \frac12|\na D_\delta(\rho_\eps)|^2 + \frac32\chi^2|\na c_\eps|^2.
\end{align*}
Then we can estimate \eqref{3.F} as
\begin{align*}
  \int_\Omega F_\delta(\rho_\eps)dx\Big|_0^t
  \le -\frac12\int_0^t\int_\Omega|\na D_\delta(\rho_\eps)|^2 dxds 
  + \frac32\chi^2\int_0^t\int_\Omega|\na c_\eps|^2 dxds.
\end{align*}
We know from the proof of Theorem \ref{thm.ex} that $(c_\eps)$ is bounded in $L^2(0,T;H^1(\Omega))$ and, moreover, that $F_\delta(\rho_\eps)$ is uniformly bounded in $(\delta,\eps)$. Then the previous inequality shows that $\na D_\delta(\rho_\eps)$ is bounded uniformly in $(\delta,\eps)$. We deduce from $1<m\le 2$ that $C'_\delta(\rho_\eps)\le \sqrt{1+2\delta}D'_\delta(\rho_\eps)\le \sqrt{3}D'_\delta(\rho_\eps)$, which yields the uniform bound
\begin{align*}
  \int_0^t\int_\Omega|\na C_\delta(\rho_\eps)|^2 dxds
  \le 3\int_0^t\int_\Omega|\na D'_\delta(\rho_\eps)|^2 dxds \le C.
\end{align*}

Since $\rho_\eps\to \rho$ pointwise a.e.\ in $\Omega_T$ and both $C_\delta$, $D_\delta$ are continuous and bounded uniformly in $\delta$, we can apply the dominated convergence theorem to conclude that
\begin{align*}
  C_\delta(\rho_\eps)\to \rho, \quad
  D_\delta(\rho_\eps)\to \frac{2}{m}\rho^{m/2}
  \quad\mbox{pointwise a.e. as }(\delta,\eps)\to 0.
\end{align*}
Thus, it follows from the uniform bounds for $\na C_\delta(\rho_\eps)$ and $\na D_\delta(\rho_\eps)$ in $L^2(\Omega_T)$ that
\begin{align*}
  \na C_\delta(\rho_\eps)\rightharpoonup \na\rho, \quad
  \na D_\delta(\rho_\eps)\rightharpoonup \frac{2}{m}\na\rho^{m/2}
  \quad\mbox{weakly in }L^2(\Omega_T).
\end{align*}
Passing to the limit $(\delta,\eps)\to 0$ in \eqref{3.F} and applying weakly lower semicontinuity, we have shown \eqref{3.claim}, which finishes the proof.
\end{proof}

\begin{lemma}[Gradient regularity for the density]
Let $1< m\le 2$. Then there exists $C>0$ only depending on the initial data such that
\begin{align*}
  \|\rho\|_{L^2(0,T;H^1(\Omega))}\le C.
\end{align*}
\end{lemma}

\begin{proof}
The condition $m\le 2$ and the property $\rho\le 1$ imply that $|\na\rho|^2\le\rho^{m-2}|\na\rho|^2$. Hence, we infer from the entropy inequality \eqref{3.entineq} and $\rho$, $c\le 1$ that
\begin{align*}
  \int_\Omega& h(\rho,c)dx\Big|_0^t + \int_0^t\int_\Omega
  \bigg(|\na\rho|^2 + \frac{\lambda}{\tau}
  (\eta|\na c|^2 + c^2)\bigg)dxds \\
  &\le \frac{1}{2}\int_0^t\int_\Omega|\na\rho|^2 dxds
  + \frac{\chi^2}{2}\int_0^t\int_\Omega|\na c|^2 dxds.
\end{align*}
The first term on the right-hand side can be absorbed by the left-hand side, and we know from \eqref{2.cc} that the second term is bounded. 
\end{proof}

We are now ready to prove Theorem \ref{thm.wsu}. 

\begin{proof}[Proof of Theorem \ref{thm.wsu}]
We split the relative entropy \eqref{3.relent} into three parts:
\begin{align*}
  H(\rho,c|\bar\rho,\bar{c})
  &= \int_\Omega\big(\rho(\log\rho-1) + (1-\rho)(\log(1-\rho)-1)
  \big)dx \\
  &\phantom{xx}+ \int_\Omega\big(-\rho\log\bar{\rho}
  - (1-\rho)\log(1-\bar\rho) + 1\big)dx
  + \frac{\lambda}{2}\int_\Omega(c-\bar{c})^2 dx.
\end{align*}
Then
\begin{align*}
  H(\rho,c|\bar\rho,\bar{c})\Big|_0^t
  &= \int_0^t\frac{d}{ds}H(\rho,c|\bar\rho,\bar{c})ds 
  = \int_\Omega\big(\rho(\log\rho-1) + (1-\rho)(\log(1-\rho)-1)
  \big)dx\Big|_0^t \\
  &\phantom{xx}- \int_0^t\big\langle\pa_t\rho,
  \log\bar{\rho}-\log(1-\bar\rho)\big\rangle ds
  + \int_0^t\bigg\langle \pa_t\bar\rho,-\frac{\rho}{\bar\rho}
  + \frac{1-\rho}{1-\bar\rho}\bigg\rangle ds \\
  &\phantom{xx}+ \lambda\int_0^t\langle\pa_t(c-\bar{c}),
  c-\bar{c}\rangle ds.
\end{align*}
Observe that the expressions in the second and third integrals on the right-hand side are well defined thanks to the lower and upper bounds for $\bar\rho$. The expression in the third integral exists since $\rho\in L^2(0,T;H^1(\Omega))$. We use inequality \eqref{3.claim} for the first integral on the right-hand side, replace the gradients of $\rho$ according to
\begin{align*}
  \rho^{m-2}|\na\rho|^2 &= (1-\rho)\rho^{m-1}\na(\log\rho-\log(1-\rho))
  \cdot\na\rho \quad\mbox{and} \\
  \na\rho &= \rho(1-\rho)\na(\log\rho-\log(1-\rho)),
\end{align*}
and insert equations \eqref{1.rho} and \eqref{1.c} in the remaining integrals. Then a computation gives
\begin{align}
  H(\rho,c|\bar\rho,\bar{c})\Big|_0^t
  &= -\int_0^t\int_\Omega\big((1-\rho)\rho^{m-1}\na\rho
  - \chi\rho(1-\rho)\na c\big)\cdot\na\bigg(\log\frac{\rho}{\bar\rho}
  - \log\frac{1-\rho}{1-\bar\rho}\bigg)dxds \nonumber \\
  &\phantom{xx}+ \int_0^t\int_\Omega
  \big((1-\bar\rho)\bar{\rho}^{m-1}\na\bar\rho
  - \chi\bar\rho(1-\bar\rho)\na \bar{c}\big)
  \cdot\na\bigg(\frac{\rho}{\bar\rho}
  - \frac{1-\rho}{1-\bar\rho}\bigg)dxds \label{3.H} \\
  &\phantom{xx}- \frac{\lambda}{\tau}\int_0^t\int_\Omega\big(
  \eta|\na(c-\bar{c})|^2 + (c-\bar{c})^2 - (\rho-\bar\rho)(c-\bar{c})
  \big)dxds. \nonumber 
\end{align}
We abbreviate
\begin{align*}
  L := (1-\rho)\na\bigg(\log\frac{\rho}{1-\rho}
  - \log\frac{\bar\rho}{1-\bar\rho}\bigg).
\end{align*}
Then the identity
\begin{align*}
  \rho\na\log\frac{\rho}{\bar\rho} 
  + (1-\rho)\na\log\frac{1-\rho}{1-\bar\rho}
  = -(\rho-\bar\rho)\na\log\frac{\bar\rho}{1-\bar\rho}
\end{align*}
implies that
\begin{align*}
  \na\log\frac{\rho}{\bar\rho} = L + \rho\na\log\frac{\rho}{\bar\rho} 
  + (1-\rho)\na\log\frac{1-\rho}{1-\bar\rho}
  = L - (\rho-\bar\rho)\na\log\frac{\bar\rho}{1-\bar\rho}.
\end{align*}
Equation \eqref{3.H} becomes
\begin{align*}
  H&(\rho,c|\bar\rho,\bar{c})\Big|_0^t
  = -\int_0^t\int_\Omega\rho^m L\cdot\na\log\frac{\rho}{\bar\rho}
  dxds \\
  &- \int_0^t\int_\Omega\bigg\{\rho^m L 
  - (1-\bar\rho)\bar{\rho}^m
  \na\bigg(\frac{\rho}{\bar\rho} - \frac{1-\rho}{1-\bar\rho}\bigg)
  \bigg\}\cdot\na\log\bar\rho dxds \\
  &+ \chi\int_0^t\int_\Omega \rho L\cdot\na(c-\bar{c})dxds
  + \chi\int_0^t\int_\Omega\bigg\{\rho L - \bar\rho(1-\bar\rho)
  \na\bigg(\frac{\rho}{\bar\rho} - \frac{1-\rho}{1-\bar\rho}\bigg)
  \bigg\}\cdot\na\bar{c} dxds \\
  &- \frac{\lambda}{\tau}\int_0^t\int_\Omega\big(\eta|\na(c-\bar{c})|^2
  + (c-\bar{c})^2 - (\rho-\bar\rho)(c-\bar{c})\big)dxds.
\end{align*}
A computation leads to
\begin{align}\label{3.H2} 
  & H(\rho,c|\bar\rho,\bar{c})\Big|_0^t
  = I_1 + \cdots + I_5, \quad\mbox{where} \\
  & I_1 = -\int_0^t\int_\Omega\bigg(\rho^m |L|^2 
  - \rho^m(\rho-\bar\rho) L\cdot\na\log\frac{\bar\rho}{1-\bar\rho}
  \bigg)dxds, \nonumber \\
  & I_2 = -\int_0^t\int_\Omega\bigg(\rho^m L - (1-\bar\rho)
  \bar{\rho}^{m-1}\rho\na\log\frac{\rho}{\bar\rho}
  + (1-\rho)\bar{\rho}^m\na\log\frac{1-\rho}{1-\bar\rho}\bigg)
  \cdot\na\log\bar\rho dxds, \nonumber \\
  & I_3 = \chi\int_0^t\int_\Omega\bigg(\rho L - (1-\bar\rho)
  \rho\na\log\frac{\rho}{\bar\rho} + (1-\rho)\bar\rho
  \na\log\frac{1-\rho}{1-\bar\rho}\bigg)\cdot\na\bar{c} dxds, 
  \nonumber \\
  & I_4 = -\frac{\lambda}{\tau}\int_0^t\int_\Omega
  \big(\eta|\na(c-\bar{c})|^2
  + (c-\bar{c})^2 - (\rho-\bar\rho)(c-\bar{c})\big)dxds, 
  \nonumber \\
  & I_5 = \chi\int_0^t\int_\Omega \rho L\cdot\na(c-\bar{c})dxds.
  \nonumber 
\end{align}
By Young's inequality with $\delta>0$,
\begin{align*}
  I_1 \le \int_0^t\int_\Omega(-\rho^m + \delta\rho^{2m})|L|^2 dxds
  + C(\delta,\bar\rho)\int_0^t\int_\Omega(\rho-\bar\rho)^2 dxds,
\end{align*}
where $C(\delta,\bar\rho)>0$ depends on the $L^\infty(\Omega_T)$ norm of $\na \bar\rho$ and the (positive) infimum of $\bar\rho$. Similarly, we estimate $I_4+I_5$:
\begin{align*}
  I_4+I_5 &\le \int_0^t\int_\Omega\bigg\{\delta\rho^2|L|^2
  + \bigg(C(\delta,\chi)-\frac{\lambda}{\tau}\eta\bigg)
  |\na(c-\bar{c})|^2 \\
  &\phantom{xx}
  - \frac{\lambda}{2\tau}(c-\bar{c})^2
  + \frac{\lambda}{2\tau}(\rho-\bar\rho)^2\bigg\}dxds.
\end{align*}
Inserting the definition of $L$, formulated as
\begin{align*}
  L = (1-\rho)\na\log\frac{\rho}{\bar\rho}
  - (1-\rho)\na\log\frac{1-\rho}{1-\bar\rho},
\end{align*}
and observing that
\begin{align*}
  \rho(\rho-\bar\rho)\na\log\frac{\rho}{\bar\rho}
  + (1-\rho)(\rho-\bar\rho)\na\log\frac{1-\rho}{1-\bar\rho}
  = -(\rho-\bar\rho)^2\na\log\frac{\bar\rho}{1-\bar\rho},
\end{align*}
we can reformulate $I_2$ according to
\begin{align*}
  I_2 &= -\int_0^t\int_\Omega(\rho^{m-1}-\bar{\rho}^{m-1})\rho L
  \cdot\na\log\bar\rho dxds \\
  &\phantom{xx}+ \int_0^t\int_\Omega\bar{\rho}^{m-1}
  \bigg(\rho(\rho-\bar\rho)\na\log\frac{\rho}{\bar\rho}
  + (1-\rho)(\rho-\bar\rho)\na\log\frac{1-\rho}{1-\bar\rho}\bigg)
  dxds \\
  &= -\int_0^t\int_\Omega(\rho^{m-1}-\bar{\rho}^{m-1})\rho L
  \cdot\na\log\bar\rho dxds \\
  &\phantom{xx}- \int_0^t\int_\Omega\bar{\rho}^{m-1}(\rho-\bar\rho)^2
  \na\log\frac{\bar\rho}{1-\bar\rho}\cdot\na\log\bar\rho dxds.
\end{align*}
Next, we exploit the fact that $1\le m\le 2$ to find that
\begin{align*}
  |\rho^{m-1}-\bar\rho^{m-1}| &= \frac{1}{m-1}\int_0^1
  (\theta\rho+(1-\theta)\bar\rho)^{m-2}d\theta|\rho-\bar\rho| \\
  &\le \frac{\bar{\rho}^{m-2}}{m-1}\int_0^1
  (1-\theta)^{m-2}d\theta|\rho-\bar\rho|
  \le C(\bar\rho)|\rho-\bar\rho|,
\end{align*}
where $C(\bar\rho)>0$ depends on the (positive) infimum of $\bar\rho$. Therefore, by the Cauchy--Schwarz inequality with $\delta>0$,
\begin{align*}
  I_2 &\le \delta\int_0^t\int_\Omega\rho^2|L|^2 dxds
  + C(\delta,\bar\rho)\int_0^t\int_\Omega
  |\rho^{m-1}-\bar{\rho}^{m-1}|^2 dxds
  + C(\bar\rho)\int_0^t\int_\Omega(\rho-\bar\rho)^2 dxds \\
  &\le \delta\int_0^t\int_\Omega\rho^2 |L|^2 dxds
  + C(\delta,\bar\rho)\int_0^t\int_\Omega|\rho-\bar{\rho}|^2 dxds,
\end{align*}
and the constants depend on the $L^\infty(\Omega_T)$ norm of $\na\bar\rho$ and the (positive) infimum of $\bar\rho$. Since the term $I_3$ equals $I_2$ with $m=1$, we can estimate this expression in a similar way. Combining the estimates, we conclude from \eqref{3.H2} that
\begin{align*}
  H&(\rho,c|\bar\rho,\bar{c})\Big|_0^t
  \le \int_0^t\int_\Omega(-\rho^m + \delta\rho^{2m} + 2\delta\rho^2)
  |L|^2 dxds \\
  &+ \int_0^t\int_\Omega\bigg(C(\delta,\chi)-\frac{\lambda}{\tau}\bigg)
  |\na(c-\bar{c})|^2 dxds
  + C(\delta,\bar\rho,\lambda/\tau)\int_0^t\int_\Omega(\rho-\bar\rho)^2 dxds.
\end{align*}
It follows from $m\le 2$ and $\rho\le 1$ that
\begin{align*}
  -\rho^m + \delta\rho^{2m} + 2\delta\rho^2
  \le -\rho^m + \delta\rho^m + 2\delta\rho^m
  = (-1+3\delta)\rho^m\le 0,
\end{align*}
if we choose $\delta\le 1/3$. Then taking $\lambda\ge \tau C(\delta,\chi)$, we find that
\begin{align*}
  H(\rho,c|\bar\rho,\bar{c})\Big|_0^t
  \le C(\delta,\bar\rho,\lambda/\tau)
  \int_0^t\int_\Omega(\rho-\bar\rho)^2 dxds.
\end{align*}
If the following lower bound for the relative entropy
\begin{align}\label{3.lower}
  H(\rho,c\bar\rho,\bar{c}) \ge 2\int_\Omega(\rho-\bar\rho)^2 dx
\end{align}
holds, we can apply Gronwall's inequality to conclude that $H((\rho,c)(t)|(\bar\rho,\bar{c})(t))=0$ for all $t\ge 0$, since the initial data are the same. This shows that $\rho(t)=\bar\rho(t)$ and $c(t)=\bar{c}(t)$ for $t\ge 0$. It remains to verify \eqref{3.lower}. Let $h_0(s) = s(\log s-1) + (1-s)(\log(1-s)-1)$. Then $h''(s)=1/(s(1-s))\ge 4$ and, by Taylor's expansion, for some $\xi$ between $\rho$ and $\bar\rho$,
\begin{align*}
  H(\rho,c|\bar\rho,\bar{c})
  &= \int_\Omega\big(h_0(\rho) - h_0(\bar\rho) 
  - h'_0(\bar\rho)(\rho-\bar\rho)\big)dx 
  + \frac{\lambda}{2}\int_\Omega(c-\bar{c})^2 dx \\
  &\ge \frac12\int_\Omega h_0''(\xi)(\rho-\bar\rho)^2dx 
  \ge 2\int_\Omega(\rho-\bar\rho)^2 dx,
\end{align*}
which shows the claim and completes the proof.
\end{proof}


\section{Steady states and pattern formation}\label{sec.steady}

We consider in this section stationary solutions. We show that they are unique if $1<m\le 2$ (and $\chi\le 1$) and that they are not unique if $m>2$. In the latter situation, there exist increasing solutions, which indicates pattern formation. To simplify the notation, we set $\tau=1$ and $\eta=1$. 

\subsection{Uniqueness of the steady state for $1<m\le 2$}

We define a steady state $(\rho,c)$ to \eqref{1.rho}--\eqref{1.bic} as a solution to
\begin{align}\label{4.srho}
  (1-\rho)\rho^{m-1}\na\rho = \chi\rho(1-\rho)\na c, \quad
  \Delta c-c+\rho = 0\quad\mbox{in }\Omega,
\end{align}
with the boundary conditions $(1-\rho)\rho^{m-1}\na\rho\cdot\nu = \na c\cdot\nu = 0$ on $\pa\Omega$. Let
\begin{align*}
  \varphi(\rho) = \frac{\rho^{m-1}}{m-1} \quad\mbox{for }\rho>0.
\end{align*}
Then, if $0<\rho<1$ in $\Omega$, we can divide \eqref{4.srho} by $\rho(1-\rho)$, yielding $\na\varphi(\rho) = \rho^{m-2}\na\rho = \chi\na c$ or $\varphi(\rho) = \chi(c+\lambda)$ for some constant $\lambda\in\R$. Therefore, any nonnegative steady state $(\rho,c)$ with vanishing flux and the property $0<\rho<1$ solves
\begin{align}
  -\Delta c + c = \rho := \varphi^{-1}(\chi(c+\lambda))
  \quad\mbox{in }\Omega, \quad \na c\cdot\nu = 0 \quad
  \mbox{on }\pa\Omega. \label{4.steady}
\end{align}
Because of $0<\rho=\varphi^{-1}(\chi(c+\lambda))<1$, the parameter $\lambda$ needs to satisfy the bounds $0<\chi(c+\lambda)<1/(m-1)$. Below we construct a solution $(c,\lambda)$ to this problem in one space dimension. Let $M:=\int_\Omega\rho dx$ be the total mass. We claim that, under certain assumptions, $(\rho,c)=(M/|\Omega|,M/|\Omega|)$ is the unique steady state. 

\begin{proposition}\label{prop.unique}
Let $1<m\le 2$ and $\chi\le 1$, or let $m>2$ and $\chi>0$ be sufficiently small. Then problem \eqref{4.steady} admits the unique solution $c=M/|\Omega|$. Consequently, for sufficiently small $\chi$, system \eqref{1.rho}--\eqref{1.bic} has the unique steady state $(\rho,c)=(M/|\Omega|,M/|\Omega|)$. 
\end{proposition}

\begin{proof}
To simplify the computations, we set $|\Omega|=1$. Clearly, $(\rho,c)=(M,M)$ is a steady state to \eqref{4.steady}, i.e., there exists $\lambda_M$ such that
\begin{align*}
  \frac{M^{m-1}}{m-1} = \varphi(M) = \chi(M+\lambda_M).
\end{align*}
Let first $\lambda\le\lambda_M$ and let $(\rho,c)$ be a solution to \eqref{4.steady}. 
We prove $c\le M$, which requires the smallness condition on $\chi$. Indeed, choosing the test function $(c-M)^+=\max\{0,c-M\}$ in \eqref{4.steady} and applying the mean-value theorem, we obtain, because of $\lambda\le\lambda_M$,
\begin{align*}
  \int_\Omega|&\na(c-M)^+|^2 dx + \int_\Omega|(c-M)^+|^2 dx \\
  &\le \int_\Omega\big(\varphi^{-1}(\chi(c+\lambda_M))
  - \varphi^{-1}(\chi(M+\lambda_M))\big)(c-M)^+ dx \\
  &= \chi\int_\Omega(\varphi^{-1})'(\xi)(c-M)^2 dx
  = \chi\int_\Omega[(m-1)\xi]^{(2-m)/(m-1)}|(c-M)^+|^2 dx,
\end{align*}
where $\xi$ is between $\chi(c+\lambda_M)$ and $\chi(M+\lambda_M)$. 

Let $1<m\le 2$. By definition of $\lambda_M$, we have $(m-1)\xi\le (m-1)\chi(M+\lambda_M) = M^{m-1}\le 1$. This implies that
\begin{align*}
  \int_\Omega|\na(c-M)^+|^2 dx + \int_\Omega|(c-M)^+|^2 dx
  \le \chi\int_{\Omega}|(c-M)^+|^2 dx,
\end{align*}
and the right-hand side can be absorbed by the left-hand side if $\chi<1$. We conclude that $(c-M)^+=0$ and hence $c\le M$ in $\Omega$. 

Next, let $m>2$. Here, we need another smallness condition on $\chi$. Indeed, it follows from $(m-1)\xi\ge(m-1)\chi\lambda_M=M^{m-1}-(m-1)\chi M$ that
\begin{align*}
  \int_\Omega&|\na(c-M)^+|^2 dx + \int_\Omega|(c-M)^+|^2 dx \\
  &\le \frac{\chi}{(M^{m-1}-(m-1)\chi M)^{(m-2)/(m-1)}}
  \int_{\Omega}|(c-M)^+|^2 dx.
\end{align*}
Assuming that
\begin{align*}
  \chi < \frac{M^{m-2}}{m-1} \quad\mbox{and}\quad 
  \frac{\chi}{(M^{m-1}-(m-1)\chi M)^{(m-2)/(m-1)}} < 1,
\end{align*}
we find that $(c-M)^+=0$ and $c\le M$ in $\Omega$. Because of mass conservation, $M=\int_\Omega cdx\le M$, and we obtain $c=M$ a.e.\ in $\Omega$. The case $\lambda\ge \lambda_M$ is treated in a similar way.

To show that the steady state with vanishing flux to \eqref{1.rho}--\eqref{1.bic} is equivalent to problem \eqref{4.steady}, we verify that $0<\rho<1$ in $\Omega$, where $\rho:=\varphi^{-1}(\chi(c+\lambda))$. For this, we need sufficiently small values of $\chi>0$. Let $(c,\lambda)$ be a solution to \eqref{4.steady} such that $0\le\rho\le 1$ (we suppose that such a solution exists). By elliptic regularity,
\begin{align}\label{4.regul}
  \|\na c\|_{L^\infty(\Omega)} \le C\|c\|_{W^{2,p}(\Omega)}
  \le C\|\rho\|_{L^p(\Omega)} \le C \quad\mbox{for }p>d.
\end{align}
We apply the gradient to $\varphi(\rho)=\chi(c+\lambda)$ to find that $\rho^{m-1}\na\rho = \rho\na\varphi(\rho)=\chi\rho\na c$. It follows from \eqref{4.regul} and $0\le\rho\le 1$ that
\begin{align*}
  \|\na\rho^m\|_{L^\infty(\Omega)} 
  = m\|\chi\rho\na c\|_{L^\infty(\Omega)}
  \le m\chi\|\na c\|_{L^\infty(\Omega)} \le \chi C_L,
\end{align*}
where $C_L>0$ is independent of $\chi$. Hence, $\rho^m$ is Lipschitz continuous. Since $|a-b|^m\le C|a^m-b^m|$ for $m\ge 1$, we compute
\begin{align*}
  \|\rho\|_{C^{0,1/m}(\overline\Omega)}
  = \sup_{x_1\neq x_2}\frac{|\rho(x_1)-\rho(x_2)|}{|x_1-x_2|^{1/m}}
  \le C\sup_{x_1\neq x_2}\bigg(
  \frac{|\rho(x_1)^m-\rho(x_2)^m|}{|x_1-x_2|}\bigg)^{1/m}
  \le C(\chi C_L)^{1/m}. 
\end{align*}

Now, assume that $\rho(x_0)=0$ for some $x_0\in\Omega$. Then
\begin{align*}
  M &= \int_\Omega\rho(x)dx
  = \int_\Omega\frac{|\rho(x)-\rho(x_0)|}{|x-x_0|^{1/m}}
  |x-x_0|^{1/m}dx \\
  &\le \|\rho\|_{C^{0,1/m}(\overline\Omega)}
  \int_\Omega|x-x_0|^{1/m}dx \le C(\chi C_L)^{1/m}
\end{align*}
or $M^m\le C^m C_L\chi$. Thus, choosing $\chi < M^m/(C^mC_L)$, this leads to a contradiction. Therefore, $\rho>0$ in $\Omega$. Similarly, we can prove that $\rho<1$ in $\Omega$. Indeed, assume that $\rho(x_1)=1$ for some $x_1\in\Omega$. Then, using $|\Omega|=1$,
\begin{align*}
  1 - M &= \int_\Omega(1-\rho(x))dx
  = \int_\Omega\frac{|\rho(x_1)-\rho(x)|}{|x-x_0|^{1/m}}
  |x-x_0|^{1/m}dx \\
  &\le \|\rho\|_{C^{0,1/m}(\overline\Omega)}\int_\Omega|x-x_1|^{1/m}dx
  \le C(\chi C_L)^{1/m},
\end{align*}
which leads to $(1-M)^m\le C^m C_L\chi$, giving a contradiction if $\chi<(1-M)^m/(C^mC_L)$. We conclude that for sufficiently small $\chi$, the pair $(\rho,c)=(M,M)$ is the unique steady state to \eqref{1.rho}--\eqref{1.bic}.
\end{proof}


\subsection{Pattern formation for $m>2$} 

Pattern formation is understood here as the existence of an increasing one-dimensional steady state. In such a situation, the steady solution is not unique. Let $\Omega=(0,1)$. We know that the steady state is characterized by finding $(c,\lambda)$ such that
\begin{align}\label{4.1d}
  -c'' + c = \varphi^{-1}(\chi(c+\lambda))\quad\mbox{for }x\in(0,1), 
  \quad c'(0) = c'(1) = 0.
\end{align}
We multiply \eqref{4.1d} by $c'$ and integrate. Then there exists a constant $\mu$ such that
\begin{equation}\label{4.G}
\begin{aligned}
  & (c')^2 = G_\lambda(c) - \mu, \quad -\lambda\le c\le 1,
  \quad\mbox{where} \\
  & G_\lambda(c) = c^2 - 2\int_{-\lambda}^c 
  \varphi^{-1}(\chi(z+\lambda))dz, \quad
  G_\mu(c(0))=G_\mu(c(1))=\mu, 
\end{aligned}
\end{equation}
where the last two equations come from the Neumann boundary conditions. We reformulate \eqref{4.G} (choosing the positive sign):
\begin{equation}\label{4.G2}
\begin{aligned}
  & c'(x) = \sqrt{G_\lambda(c(x))-\mu}\quad\mbox{for }x\in(0,1), \quad
  c(0)=c_-, \quad c(1)=c_+, \\
  & G_\lambda(c_\pm)=\mu, \quad G_\lambda(z)>\mu\quad\mbox{for }
  z\in(c_-,c_+).
\end{aligned}
\end{equation}
Notice that $c'(x)>0$ implies that $c_-<c_+$, so the interval $(c_-,c_+)$ is well defined. 

Next, we construct an increasing solution to \eqref{4.G2} satisfying the Neumann boundary conditions at $x=0$ and $x=X(\lambda,\mu)$,
\begin{align*}
  c'(x)>0\quad\mbox{for }x\in(0,X(\lambda,\mu)), \quad
  c(0) = c_-, \quad c(X(\lambda,\mu)) = c_+. 
\end{align*}
The aim is to find $(\lambda,\mu)$ such that $X(\lambda,\mu)=1$ (Figure \ref{fig.inc}). We integrate \eqref{4.G2}:
\begin{align*}
  & X(\lambda,\mu) = \int_0^{X(\lambda,\mu)}
  \frac{c'(x)dx}{\sqrt{G_\lambda(c(x))-\mu}}
  = \int_{c_-}^{c_+}\frac{dz}{\sqrt{G_\lambda(z)-\mu}} 
  =: X_1(\lambda,\mu) + X_2(\lambda,\mu), \\
  &\mbox{where}\quad 
  X_1(\lambda,\mu) = \int_{c_-}^{\widetilde{c}}
  \frac{dz}{\sqrt{G_\lambda(z)-\mu}}, \quad
  X_2(\lambda,\mu) = \int_{\widetilde{c}}^{c_+}
  \frac{dz}{\sqrt{G_\lambda(z)-\mu}},
\end{align*}
and $\widetilde{c}\in(c_-,c_+)$ satisfies $G'_\lambda(\widetilde{c})=0$ and $G''_\lambda(\widetilde{c})<0$. Such a value exists, since $G_\lambda(c_-)=G_\lambda(c_+)$ and $G_\lambda(c)\ge \mu$ for $c\in(c_-,c_+)$. 

\begin{figure}[ht]
\includegraphics[width=60mm]{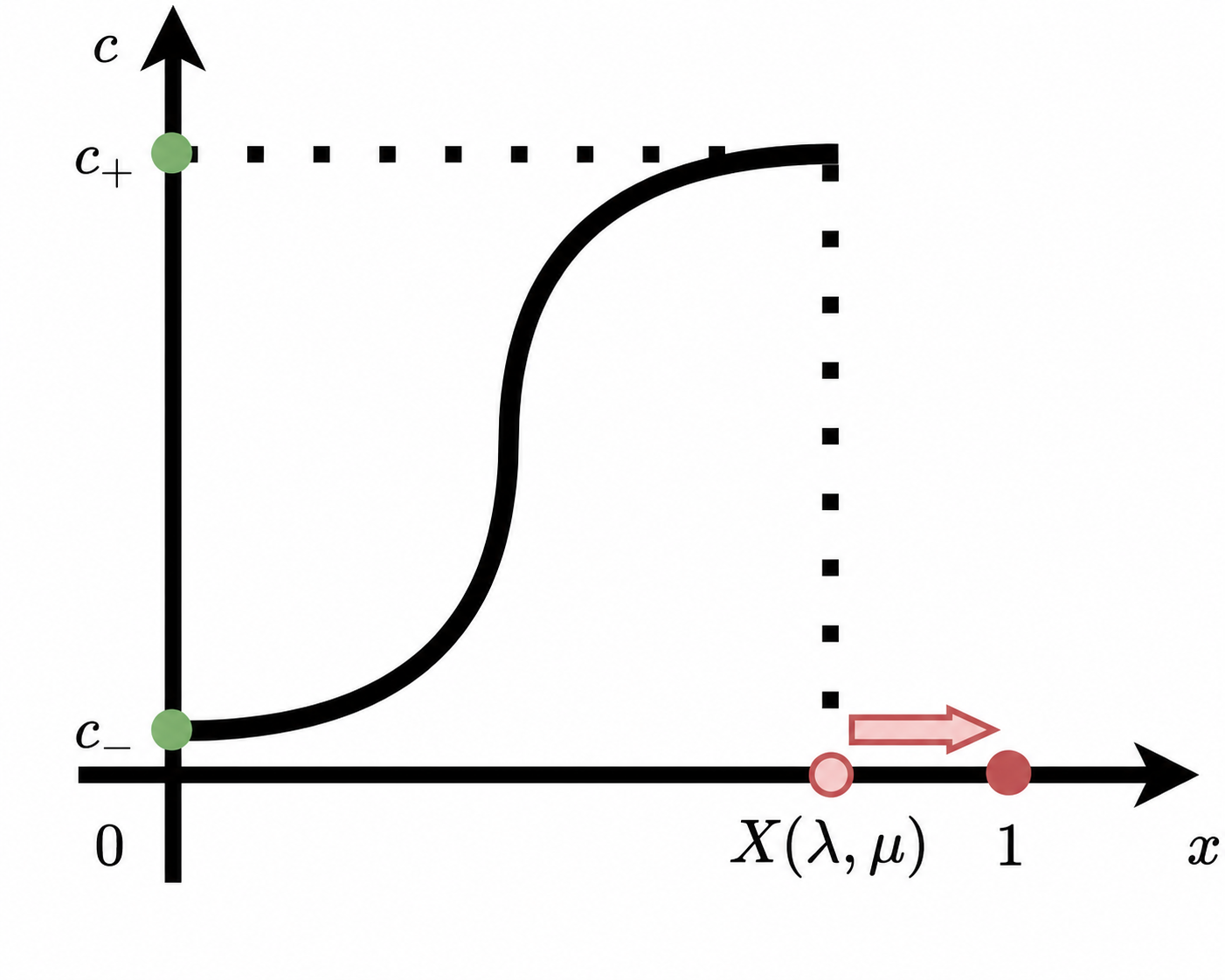}
\caption{Schematic profile of an increasing solution $c(x)$ in $[0,X(\lambda,\mu)]$.}
\label{fig.inc}
\end{figure}

Our strategy is to construct first an increasing solution $c$ on $[0,X(\lambda,\mu)]$ and then to find a pair $(\lambda^*,\mu^*)$ satisfying $X(\lambda^*,\mu^*)=1$ as well as the mass conservation condition
\begin{align*}
  M(\lambda^*,\mu^*) := \int_0^{X(\lambda^*,\mu^*)}\rho(x)dx
  = \int_0^{X(\lambda^*,\mu^*)}\varphi^{-1}(\chi(c(x)+\lambda^*))dx.
\end{align*}

\begin{proposition}[Existence of increasing solutions]\label{prop.inc}
Let $m>2$ and $0<\chi<1/(m-1)$. Then there exists a solution $(c,\lambda^*,\mu^*)$ to \eqref{4.G2} satisfying $c>0$ and $c'>0$ in $(0,1)$. Consequently, $(c,\lambda^*)$ solves \eqref{4.1d} and the mass satisfies $0<M<1$.
\end{proposition}

\begin{figure}[ht]
\includegraphics[width=70mm]{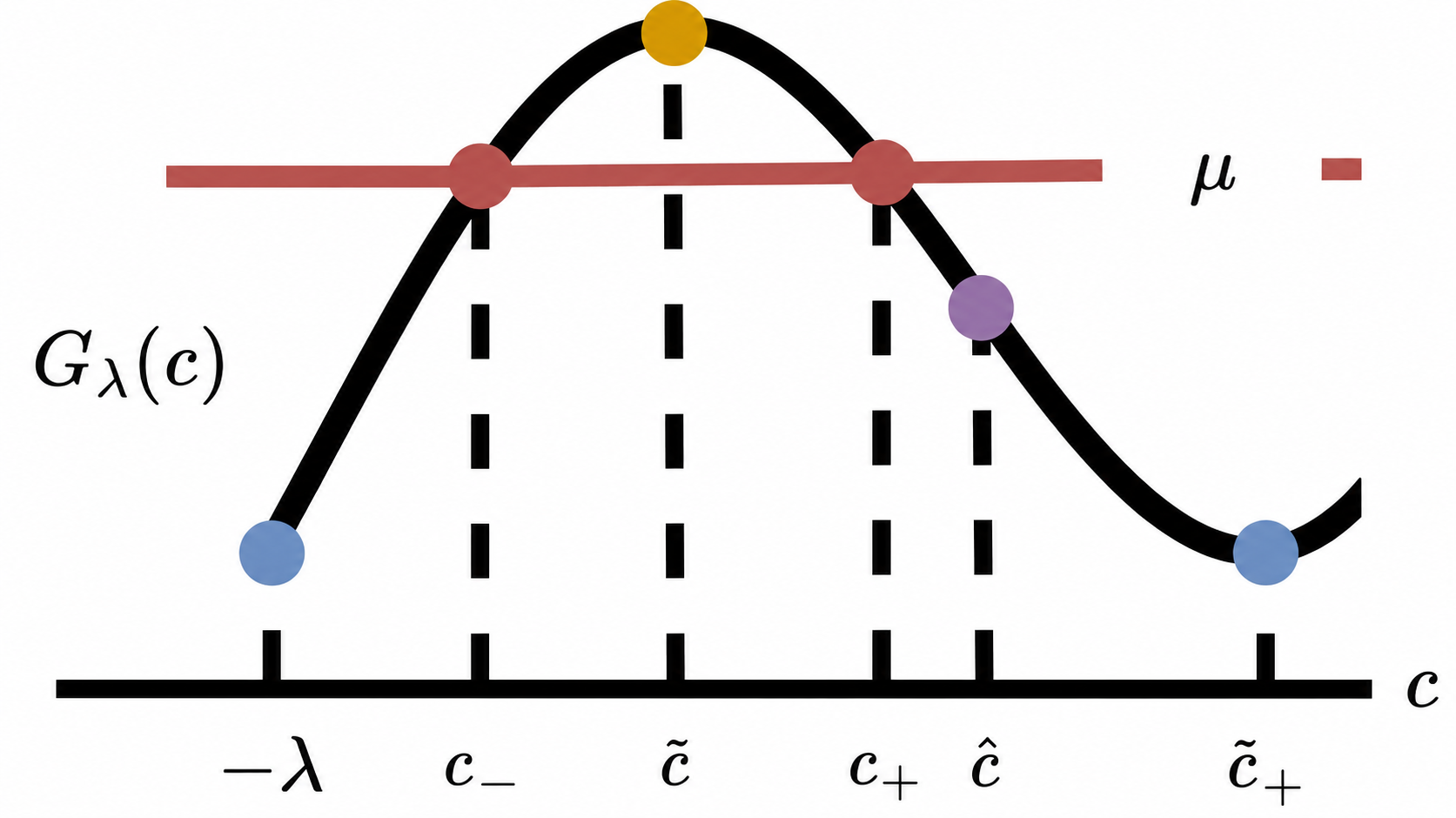}
\includegraphics[width=70mm]{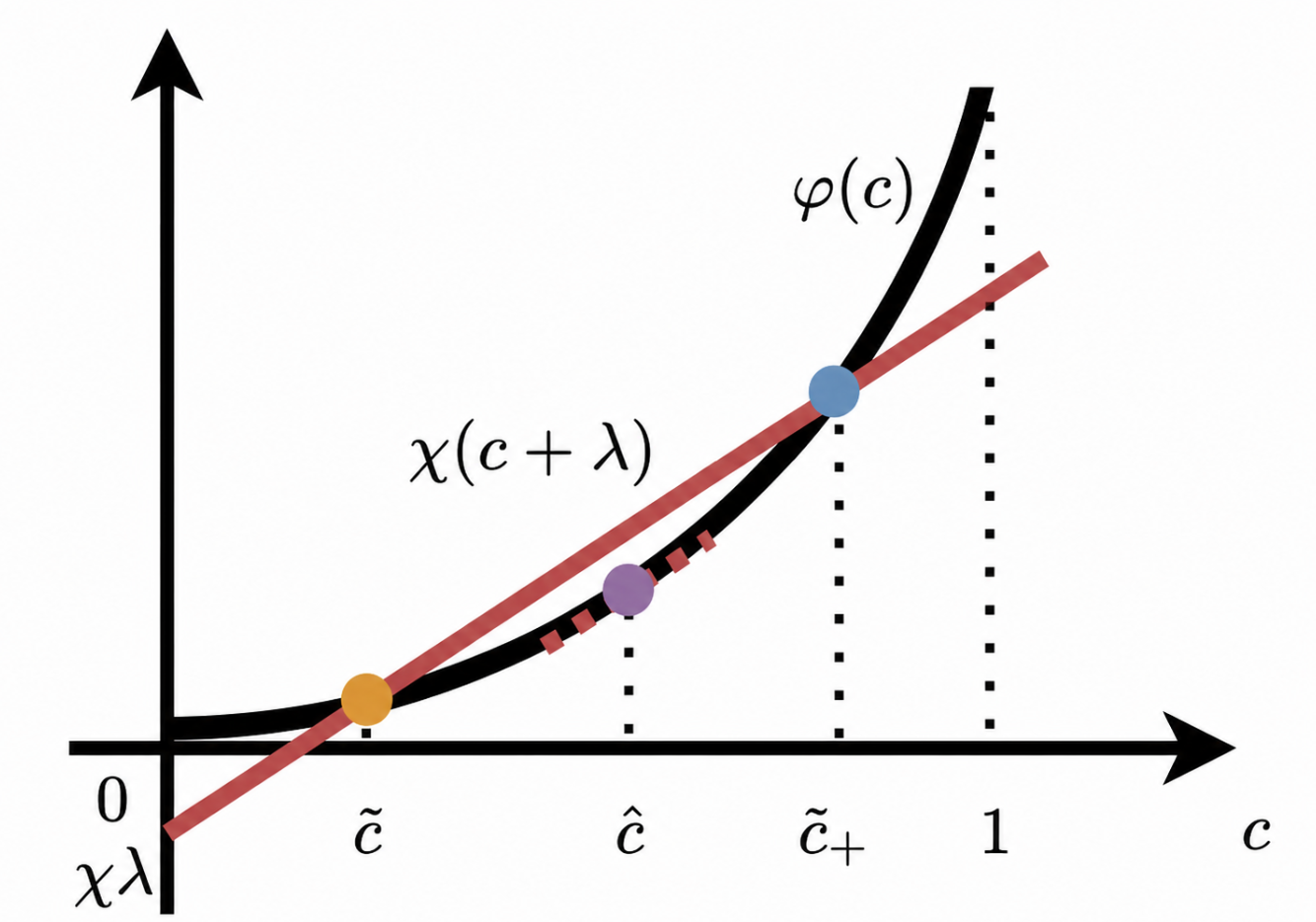}
\caption{Left: Profile of the function $G_\lambda$. Right: Intersection points of $\varphi$ and $c\mapsto\chi(c+\lambda))$.}
\label{fig.Glambda}
\end{figure}

Before we prove this result, we study the profile of $G_\lambda$. We determine the parameter regime in which $G_\lambda$ possesses two extremas $\widetilde{c}$ and $\widetilde{c}_+$; see Figure \ref{fig.Glambda} (left). It holds $0=G'_\lambda(c) = 2c - 2\varphi^{-1}(\chi(c+\lambda))$ if and only if
\begin{align}\label{4.lambda}
  \frac{c^{m-1}}{m-1} = \varphi(c) = \chi(c+\lambda).
\end{align}
This problem may have one or two solutions in $[0,1]$, depending on the value of $\lambda$; see Figure \ref{fig.lambda} (right). The line $c\mapsto \chi(c+\lambda)$ may intersect with the graph of $\varphi$ only once. This happens if $\varphi'(c)=\chi$, giving $c=\chi^{1/(m-2)}$. Inserting this value into \eqref{4.lambda} leads to 
\begin{align*}
  \frac{\chi^{(m-1)/(m-2)}}{m-1} = \chi(\chi^{1/(m-2)} + \lambda)
  \quad\mbox{and thus}\quad \lambda = -\frac{m-2}{m-1}\chi^{1/(m-2)} < 0.
\end{align*}
On the other hand, if the line $c\mapsto \chi(c+\lambda)$ is shifted upwards as in Figure \ref{fig.lambda} (right), there are two intersection points. Let one intersection point be equal to zero. The abscissa of the second intersection point may be larger than one, which is outside of the interval $[0,1]$ and which does not provide a restriction on the value of $\lambda$. Let the abscissa equal to one (or smaller) and let $\chi\le 1/(m-1)$. Then $0<c=\varphi^{-1}(\chi(c+\lambda)) < 1$ implies that $1/(\chi(m-1)) > 1+\lambda$ and $\lambda<0$. Thus, we choose $\lambda$ satisfying
\begin{align*}
  -\frac{m-2}{m-1}\chi^{1/(m-2)} < \lambda 
  < \min\bigg\{0,\frac{1}{(m-1)\chi}-1\bigg\} = 0.
\end{align*}
For $\lambda$ from this interval, there exist two solutions $\widetilde{c}$ and $\widetilde{c}_+$ to \eqref{4.lambda} satisfying
$0 < \widetilde{c} < \widetilde{c}_+ < 1$. These two values are critical points of $G_\lambda$; see Figure \ref{fig.lambda} (left).

\begin{figure}[ht]
\includegraphics[width=140mm]{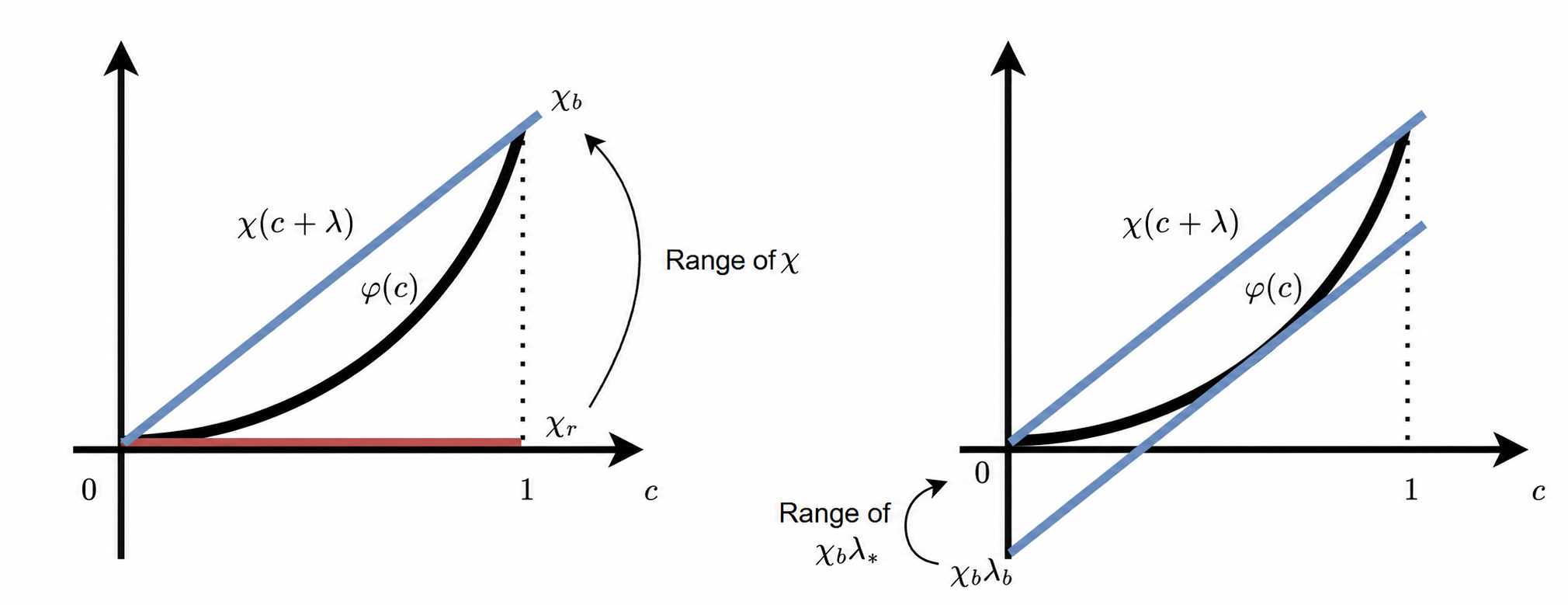}
\caption{Illustration of the solution to \eqref{4.lambda} with $\chi_r=0$ and $\chi_b=1/(m-1)$.}
\label{fig.lambda}
\end{figure}

We determine $\widehat{c}$ such that
\begin{align*}
  0 = G''_\lambda(c)
  = 2 - 2\chi(\varphi^{-1})'(\chi(c+\lambda))
  = 2\bigg(1 - \frac{\chi}{\varphi'(\rho)}\bigg),
\end{align*} 
recalling that $\rho = \varphi^{-1}(\chi(c+\lambda))$. Since $\varphi'(0)=0$, $\varphi'(1)=1$, and $\varphi'$ is increasing, there exists a unique $\widehat\rho$ such that $\varphi'(\widehat{\rho})=\chi$. Then, defining $\widehat{c}$ by $\widehat\rho = \varphi^{-1}(\chi(\widehat{c}+\lambda))$, it holds that $ G''_\lambda(\widehat{c})=0$. Since $\varphi'$ is positive, we have
\begin{align*}
  & G''_\lambda(c) < 0 \quad\mbox{if and only if}\quad
  \varphi'(\rho) < \chi\quad\mbox{if and only if}\quad
  \rho<\widehat{\rho},\ c<\widehat{c}, \\
  & G''_\lambda(c) > 0 \quad\mbox{if and only if}\quad
  \varphi'(\rho) > \chi\quad\mbox{if and only if}\quad
  \rho>\widehat{\rho},\ c>\widehat{c}.
\end{align*}
We know already that $G''_\lambda(\widetilde{c})<0$. Furthermore, the inequality $\varphi'(\rho) > \chi$ is satisfied by $\widetilde{c}_+$ such that $G''_\lambda(\widetilde{c}_+)>0$. This situation is illustrated in Figure \ref{fig.Glambda} (left). 

We show now that there exists $\lambda^*<0$ such that $G_{\lambda^*}(-\lambda^*) = G_{\lambda^*}(\widetilde{c}_+)$.

\begin{lemma}\label{lem.propG}
Let $m>2$ and $\chi_r:=0<\chi<\chi_b:=1/(m-1)$. Then there exists $\lambda^*$ satisfying
\begin{align}\label{4.lam}
  -\frac{m-2}{m-1}\chi^{1/(m-2)} < \lambda^* 
  < \bigg(\frac{1}{(\pi^2+1)(m-1)} - 1\bigg)
  \bigg(\frac{\chi}{\pi^2+1}\bigg)^{1/(m-2)},
\end{align} 
such that $G_{\lambda^*}(-\lambda^*) \leq G_{\lambda^*}(\widetilde{c}_+)$.
\end{lemma}

\begin{remark}\rm
We verify the well-posedness of \eqref{4.lam}, i.e., we show that
\begin{align}\label{4.m}
  \frac{m-2}{m-1}
  > \bigg(1-\frac{1}{(\pi^2+1)(m-1)}\bigg)
  \bigg(\frac{1}{\pi^2+1}\bigg)^{1/(m-2)}.
\end{align}
To this end, we introduce
\begin{align*}
  F(m) &= \log\frac{m-2}{m-1} - \log\bigg\{
  \bigg(1-\frac{1}{(\pi^2+1)(m-1)}\bigg)
  \bigg(\frac{1}{\pi^2+1}\bigg)^{1/(m-2)}\bigg\} \\
  &= \log(m-2) - \log(m-1) - \log\bigg(1-\frac{1}{(\pi^2+1)(m-1)}\bigg)
  + \frac{\log(\pi^2+1)}{m-2}.
\end{align*}
Determining the values of the derivative
\begin{align*}
  F'(m) &= \frac{((\pi^2 + 2) - (\pi^2+1)m)\log(\pi^2+1)
  + \pi^2(m-2)}{(m-2)^2((\pi^2+1)m - (\pi^2+2))} \\
  &= -\frac{\log(\pi^2+1)}{(m-2)^2} 
  + \frac{\pi^2}{(m-2)((\pi^2+1)m - (\pi^2+2))}
\end{align*}
shows that $F'(m)<0$ for $m>2$. Then it follows from $F(m)\to+\infty$ as $m\to 2^+$ and $F(m)\to 0$ as $m\to\infty$ that $F(m)>0$ for all $m>2$. This shows inequality \eqref{4.m}.
\end{remark}

\begin{proof}[Proof of Lemma \ref{lem.propG}]
Let $\chi_r<\chi<\chi_b$. We introduce the function
\begin{align*}
  J(\chi,c) := G_\lambda(c) - G_\lambda(-\lambda)
  = c^2 - \lambda^2 - 2\int_{-\lambda}^c\varphi^{-1}(\chi(z+\lambda))dz.
\end{align*}
Let $\lambda=\lambda_b:=-(m-2)\chi^{1/(m-2)}/(m-1)$. This corresponds to the case when the curves $\varphi$ and $c\mapsto\chi(c+\lambda)$ intersect at one point, implying that $\chi(c+\lambda)$ is below the curve $\varphi$, i.e.\ $\varphi(c)\ge\chi(c+\lambda)$ or
$c \ge \varphi^{-1}(\lambda(c+\lambda))$ for all $c\in(-\lambda,1)$. Consequently,
\begin{align*}
  J(\chi,c) \ge c^2 - \lambda^2 - 2\int_{-\lambda}^c zdz = 0
  \quad\mbox{for all }c\in(-\lambda,1),
\end{align*}
which implies that $G_\lambda(-\lambda)\le G_\lambda(c)$. 

Let $\lambda=0$. The intersection points of $\varphi$ and $c\mapsto\chi c$ are exactly zero and $\widetilde{c}_+$. We know that $\widetilde{c}$ is an extremal point of $G_\lambda$. We find, for sufficiently small $c>0$, that
\begin{align*}
  G_0(0) - G_0(c) &= -c^2 + 2\int_0^c\varphi^{-1}(\chi z)dz
  = -c^2 + 2\int_0^c[(m-1)\chi z]^{1/(m-1)}dz \\
  &= c\bigg(-c^{(m-2)/(m-1)} + 2[(m-1)\chi]^{1/(m-1)}
  \frac{m}{m-1}\bigg) > 0,
\end{align*}
so zero is a local maximum of $G_\lambda$, and since $\widetilde{c}<\widetilde{c}_+$, we infer that $\widetilde{c}=0$. The case $\lambda=0$ corresponds to the situation when the curve $\varphi$ is below the line $c\mapsto\chi c$, i.e.\ $\varphi(c)\le\chi c$ or $c\le\varphi^{-1}(\chi c)$ for all $c\in(0,\widetilde{c}_+)$. Hence,
\begin{align*}
  G_0'(c) = 2(c - \varphi^{-1}(\chi c)) < 0\quad\mbox{for }
  c\in(0,\widetilde{c}_+)
\end{align*}
Consequently, $G_0(\widetilde{c}_+)<G_0(0)$ and $G_\lambda(-\lambda)>G_\lambda(\widetilde{c}_+)$ for $\lambda=0$. By continuity, there exists $\lambda^*$ in the desired range such that $G_{\lambda^*}(-\lambda^*)\leq G_{\lambda^*}(\widetilde{c}_+)$.
\end{proof}

We are now ready to prove Proposition \ref{prop.inc}. 

\begin{proof}[Proof of Proposition \ref{prop.inc}]
The idea is to show that there exists $\mu^*$ such that $X(\lambda^*,\mu^*)=1$. First, we prove that there exists $\mu$ such that $X(\lambda^*,\mu)<1$. Recall the boundary values $c_\pm=c_\pm(\mu)$; see \eqref{4.G2}. A Taylor expansion of $G_{\lambda^*}(c)$ around $\widetilde{c}$ shows that
\begin{align*}
  \lim_{\substack{\mu\to G_{\lambda^*}(\widetilde{c}), \\
  c\in(c_-(\mu),c_+(\mu))}}A_\mu(c)
  = -\frac12G''_{\lambda^*}(\widetilde{c}), \quad
  \mbox{where}\quad
  A_\mu(c) := \frac{G_{\lambda^*}(c)-\mu}{(c_+-c)(c-c_-)} .
\end{align*}
We set
\begin{align*}
  c(z) := \frac{z}{2}\big(c_+(\mu)-c_-(\mu)\big)
  + \frac12\big(c_+(\mu)+c_-(\mu)\big) \quad\mbox{for }z\in[-1,1].
\end{align*}
Observing that $c(\pm 1) = c_\pm(\mu)$, a substitution leads to
\begin{align*}
  X(\lambda^*,\mu) &= \int_{c_-(\mu)}^{c_+(\mu)}
  \frac{dc}{\sqrt{G_{\lambda^*}(c)-\mu}}
  = \int_{c_-(\mu)}^{c_+(\mu)}
  \frac{dc}{\sqrt{(c_+(\mu)-c)(c-c_-(\mu))A_\mu(c)}} \\
  &= \int_{-1}^1\frac{dz}{\sqrt{(1-z)(1+z)A_\mu(c(z))}}.
\end{align*}
By dominated convergence,
\begin{align*}
  \lim_{\mu\to G_{\lambda^*}(\widetilde{c})}X(\lambda^*,\mu)
  =  \lim_{\mu\to G_{\lambda^*}(\widetilde{c})}\int_{-1}^1
  \frac{\sqrt{2}dz}{\sqrt{-(1-z)(1+z)G''_{\lambda^*}(\widetilde{v})}}
  = \frac{\sqrt{2}\pi}{\sqrt{-G''_{\lambda^*}(\widetilde{v})}}.
\end{align*}

We claim that $G''_{\lambda^*}(\widetilde{c})<-2\pi^2$. Then, choosing $\mu$ smaller than but close to $G''_{\lambda^*}(\widetilde{c})$, we obtain $X(\lambda^*,\mu)<1$. To prove the claim, we recall that $G''_{\lambda^*}(\widetilde{c}) = 2(1-\chi/\varphi'(\widetilde{c}))$. Therefore, the inequality $G''_{\lambda^*}(\widetilde{c})<-2\pi^2$ is equivalent to
\begin{align}\label{4.tildec}
  (\pi^2+1)\widetilde{c}^{m-2} = (\pi^2+1)\varphi'(\widetilde{c}) 
  < \chi = \frac{\varphi(\widetilde{c})}{\widetilde{c}+\lambda^*}
  = \frac{\widetilde{c}^{m-1}}{(m-1)(\widetilde{c}+\lambda^*)}.
\end{align}
We solve this inequality for $\lambda^*$:
\begin{align*}
  \lambda^* < \widetilde{c}\bigg(\frac{1}{(\pi^2+1)(m-1)}-1\bigg).
\end{align*}
This inequality holds true in view of our assumption \eqref{4.lam} and the property $\widetilde{c} < (\chi/(\pi^2+1))^{1/(m-2)}$ from \eqref{4.tildec}.

Next, we prove that $X(\lambda^*,\mu)=+\infty$ if $\mu=G_{\lambda^*}(\widetilde{c}_+)$. We observe that in this case $c_+=\widetilde{c}_+$ and $G'_{\lambda^*}(c_+)=0$, $G''_{\lambda^*}(c_+)>0$ (see Figure \ref{fig.Glambda} left). Then, by Taylor expansion around $c_+$,
\begin{align*}
  \lim_{c\to c_+}\frac{c_+-c}{\sqrt{G_{\lambda^*}(c)
  - G_{\lambda^*}(c_+)}}
  &= \lim_{c\to c_+}\frac{c-c_+}{\sqrt{G_{\lambda^*}(c_+)(c-c_+)^2/2
  + o((c-c_+)^2)}} \\
  &= \sqrt{\frac{2}{G_{\lambda^*}''(c_+)}} > 0.
\end{align*}
Thus, the integrand of
\begin{align*}
  X(\lambda^*,\mu) = \int_{c_-}^{c_+}\frac{dz}{\sqrt{G_{\lambda^*}(z)
  - G_{\lambda^*}(c_+)}}
\end{align*}
behaves like $1/(c_+ -c)$ for $c\to c_+$, which means that $X(\lambda^*,\mu)$ diverges. This shows that $X(\lambda^*,\mu)=+\infty$, proving the claim.
\end{proof}


\section{Long-time behavior of solutions}\label{sec.long} 

We have shown in the previous section that for $1<m\le 2$ and $\chi\le 1$, the unique steady state to \eqref{1.rho}--\eqref{1.bic} is the constant pair $(M,M)$ (if $|\Omega|=1$). In this section, we show that the solution constructed in Theorem \ref{thm.ex} converges to this steady state as $t\to\infty$. 

Recall definition \eqref{1.relent} of the relative entropy $H_1$ and let $|\Omega|=1$. We deduce from \eqref{3.claim} and mass conservation that
\begin{align*}
  H_1(\rho,c|M,M)\Big|_s^t
  &= \int_\Omega\big(\rho(\log\rho-1)+(1-\rho)(\log(1-\rho)-1)\big)
  dx\Big|_s^t - \tau\int_s^t\int_\Omega\pa_t c\Delta c dxds \\
  &\le -\int_s^t\int_\Omega\rho^{m-2}|\na\rho|^2 dxds
  + (\chi+1)\int_s^t\int_\Omega\na\rho\cdot\na c dxds \\
  &\phantom{xx}
  - \int_s^t\int_\Omega(\eta(\Delta c)^2 + |\na c|^2)dxds.
\end{align*}
We use the Poincar\'e inequality 
\begin{align*}
  K_1\int_\Omega|\na c|^2 dx \le \int_\Omega|\Delta c|^2 dx, 
\end{align*}
where $K_1>0$ is the principal Neumann eigenvalue (see the proof of \cite[Theorem 2.1]{CCWWZ20}), Young's inequality with $\eps>0$, the property $\rho^{m-2}\ge 1$ (which follows from $m\le 2$), and the bound $\chi\le 1$ to estimate as follows:
\begin{align*}
  H_1&(\rho,c|M,M)\Big|_s^t
  \le -\int_s^t\int_\Omega\big(\rho^{m-2}|\na\rho|^2
  + (\eta K_1+1)|\na c|^2\big)dxd\sigma \\
  &\phantom{xx}+ (\chi+1)\int_0^t\int_\Omega\na\rho\cdot\na c 
  dxd\sigma \nonumber \\
  &\le -\int_s^t\int_\Omega\bigg\{\bigg(1-\frac{1}{1+\eps}\bigg)
  |\na\rho|^2
  + \bigg(\eta K_1+1-\frac{1+\eps}{4}(\chi+1)^2\bigg)|\na c|^2\bigg\}dxd\sigma. \nonumber 
\end{align*}
Since $\chi\le 1$, the choice $\eps=\eta K_1$ leads to
\begin{align*}
  \eta K_1+1-\frac{1+\eps}{4}(\chi+1)^2
  \ge \eta K_1 + 1 - (1+\eps) = 0
\end{align*}
and hence to
\begin{align}\label{4.H}
  H_1(\rho,c|M,M)\Big|_s^t
  \le -\bigg(1-\frac{1}{1+\eta K_1}\bigg)\int_s^t\int_\Omega
  |\na\rho|^2 dxd\sigma.
\end{align}

We use the following variant of the logarithmic Sobolev inequality \cite[Theorem 1]{AbLe25}:
\begin{align*}
  \int_\Omega\rho\log\frac{\rho}{M}dx 
  &\le C(\Omega)M\int_\Omega|\na \rho|^2 dx, \\
  \int_\Omega(1-\rho)\log\frac{1-\rho}{1-M}dx
  &\le C(\Omega)(1-M)\int_\Omega|\na\rho|^2 dx,
\end{align*}
where $C(\Omega)>1$ only depends on $d$ and the diameter of $\Omega$,
implying that
\begin{align*}
  \int_\Omega\bigg(\rho\log\frac{\rho}{M} 
  + (1-\rho)\log\frac{1-\rho}{1-M}\bigg)dx
  \le C(\Omega)\int_\Omega|\na\rho|^2 dx.
\end{align*}
Then \eqref{4.H} becomes
\begin{align*}
  H_1(\rho,c|M,M)\Big|_s^t
  \le -\mu_0\int_s^t H_1(\rho,c|M,M)dxd\sigma, \quad
  \mbox{where}\ \mu_0 = \frac{1}{C(\Omega)}\frac{\eta K_1}{1+\eta K_1},
\end{align*} 
and after division by $s-t$ and passing to the limit $s\to t$,
\begin{align*}
  \frac{d}{dt}H_1(\rho,c|M,M)\le -\mu_0 H_1(\rho,c|M,M).
\end{align*}
We conclude from Gronwall's inequality that 
\begin{align*}
  H_1(\rho(t),c(t)|M,M) \le H_1(\rho^0,c^0|M,M)e^{-\mu_0 t}, \quad t>0.
\end{align*}
A Taylor expansion around $M$ leads for some $\xi$ between $\rho$ and $M$ to
\begin{align*}
  \rho\log\frac{\rho}{M} + (1-\rho)\log\frac{1-\rho}{1-M}
  = \frac12\frac{(\rho-M)^2}{\xi(1-\xi)} \ge 2(\rho-M)^2.
\end{align*}
This shows that
\begin{align}\label{4.rho}
  2\|\rho(t)-M\|_{L^2(\Omega)}^2 \le H_1(\rho^0,c^0|M,M)e^{-\mu_0 t}.
\end{align}

The exponential decay of $c(t)$ follows from \eqref{4.rho}. Indeed, using the test function $c-M$, we find from \eqref{1.c} that
\begin{align*}
  \frac{\tau}{2}\frac{d}{dt}\int_\Omega(c-M)^2 dx
  &+ \int_\Omega\big(\eta|\na(c-M)|^2 + (c-M)^2\big)dx
  = \int_\Omega (\rho-M)(c-M)dx \\
  &\le -\frac{1}{2}\int_\Omega(c-M)^2dx 
  + \frac{1}{2}\int_\Omega(\rho-M)^2dx
\end{align*}
and hence,
\begin{align*}
  \tau\frac{d}{dt}\int_\Omega(c-M)^2 dx + \int_\Omega(c-M)^2dx
  \le \int_\Omega(\rho-M)^2dx \le \frac12H_1(\rho^0,c^0|M,M)e^{-\mu_0 t}.
\end{align*}
We apply Gronwall's inequality:
\begin{align*}
  \|c(t)-M\|_{L^2(\Omega)}^2 \le Ce^{-\min\{1/\tau,\mu_0\}t}, \quad t>0,
\end{align*}
where $C>0$ depends on $\tau$, $(\rho^0,c^0)$, and $\Omega$. With the test function $-\Delta(c-M)$, we infer that 
\begin{align*}
  \frac{\tau}{2}\frac{d}{dt}&\int_\Omega|\na(c-M)|^2 dx
  + \int_\Omega\big(\eta|\Delta(c-M)|^2 + |\na(c-M)|^2\big)dx \\
  &= -\int_\Omega(\rho-M)\Delta(c-M)dx
  \le\eta \int_\Omega|\Delta(c-M)|^2 
  + \frac{\eta}{4}\int_\Omega(\rho-M)|^2 dx.
\end{align*}
Applying Gronwall's inequality again leads to
\begin{align*}
  \|\na(c(t)-M)\|_{L^2(\Omega)}^2\le C e^{-\min\{1/\tau,\mu_0\}t}, 
  \quad t>0,
\end{align*}
finishing the proof of Theorem \ref{thm.time}.


\section{Asymptotic limits}\label{sec.asym}

In this section, we prove the asymptotic limits $\tau\to 0$ (Theorem \ref{thm.D}) and $\eta\to 0$ (Theorem \ref{thm.eta}). 

\subsection{Parabolic--elliptic limit}

We wish to perform the limit $\tau\to 0$ in equations \eqref{1.rho}--\eqref{1.c}, leading to the parabolic--elliptic Keller--Segel model. Let $(\rho_\tau,c_\tau)$ be a solution to 
\begin{equation}\label{5.rhoc}
\begin{aligned}
  \pa_t\rho_\tau &= \diver\big((1-\rho_\tau)\rho_\tau^{m-1}\na\rho_\tau
  - \chi\rho_\tau(1-\rho_\tau)\na c_\tau\big), \\
  \tau\pa_t c_\tau &= \eta\Delta c_\tau - c_\tau + \rho_\tau
  \quad\mbox{in }\Omega,\ t>0,
\end{aligned}
\end{equation}
satisfying the initial and no-flux boundary conditions \eqref{1.bic}. The proof is based on uniform estimates derived from the entropy functional \eqref{1.Ftau}. 

\begin{lemma}\label{lem.entD}
Let $(\rho_\tau,c_\tau)$ be a weak solution to \eqref{1.bic}, \eqref{5.rhoc}. Then there exists a constant $C>0$, independent of $\tau$, such that
\begin{align*}
  F_\tau(\rho_\tau,c_\tau)\Big|_0^t 
  + \frac12\int_0^t\int_\Omega \rho_\tau^{m-2}|\na\rho_\tau|^2 dxds 
  \le C\tau.
\end{align*}
\end{lemma}

\begin{proof}
We sum \eqref{3.claim}, rewritten as
\begin{align*}
  \int_\Omega\big(\rho_\tau(\log\rho_\tau-1)
  &+ (1-\rho_\tau)(\log(1-\rho_\tau)-1)\big)dx\Big|_0^t
  + \int_0^t\int_\Omega\rho_\tau^{m-2}|\na\rho_\tau|^2 dxds \\
  &\le \chi\int_0^t\int_\Omega\na\rho_\tau\cdot\na c_\tau dxds,
\end{align*}
and
\begin{align*}
  \frac{\tau\chi}{\eta}\int_\Omega\rho_\tau c_\tau dx\Big|_0^t
  &= \frac{\tau\chi}{\eta}\int_0^t\langle\pa_t\rho_\tau,c_\tau\rangle ds
  + \frac{\tau\chi}{\eta}\int_0^t\int_\Omega\rho_\tau\pa_t c_\tau dxds \\
  &=  -\frac{\tau\chi}{\eta}\int_0^t\int_\Omega\big(
  (1-\rho_\tau)\rho_\tau^{m-1}\na\rho_\tau\cdot\na c_\tau
  - \chi\rho_\tau(1-\rho_\tau)|\na c_\tau|^2\big)dxds \\
  &\phantom{xx}- \chi\int_0^t\int_\Omega
  \na\rho_\tau\cdot\na c_\tau dxds
  + \frac{\chi}{\eta}\int_0^t\int_\Omega(\rho_\tau-c_\tau)\rho_\tau dxds.
\end{align*}
Then the term involving $\na\rho_\tau\cdot\na c_\tau$ cancels and, observing that $(\rho_\tau-c_\tau)c_\tau\le 1$ and applying Young's inequality, we obtain
\begin{align*}
  F_\tau(\rho_\tau,c_\tau)\Big|_0^t 
  &+ \int_0^t\int_\Omega\rho_\tau^{m-2}|\na\rho_\tau|^2 dxds \\
  &\le \frac{1}{2}\int_0^t\int_\Omega\rho_\tau^{m-2}
  |\na\rho_\tau|^2dxds
  + \frac12\bigg(\frac{\tau\chi}{\eta}\bigg)^2
  \int_0^t\int_\Omega(1-\rho_\tau)^2\rho_\tau^m
  |\na c_\tau|^2 dxds \\
  &\phantom{xx}+ \frac{\tau\chi^2}{\eta}
  \int_0^t\int_\Omega\rho_\tau(1-\rho_\tau)
  |\na c_\tau|^2 dxds + \chi\mbox{meas}(\Omega_T).
\end{align*}
Using the test function $c_\tau$ in equation \eqref{1.c} gives
\begin{align*}
  \tau\int_\Omega c_\tau(t)^2 dx - \tau\int_\Omega (c^0)^2 dx
  + \int_0^t\int_\Omega\big(\eta|\na c_\tau|^2 + c_\tau^2\big)dxds
  = \int_0^t\int_\Omega\rho_\tau c_\tau dxds \le C,
\end{align*}
and we conclude a uniform bound for $(c_\tau)$ in $L^2(0,T;H^1(\Omega))$, finishing the proof.
\end{proof}

\begin{lemma}\label{lem.ceps}
Let $c^0\in H^1(\Omega)$. Then there exists $C>0$ independent of $\tau$ such that
\begin{align*}
  \sqrt\tau\|\pa_t c_\tau\|_{L^2(\Omega_T)} 
  + \|c_\tau\|_{L^\infty(0,T;H^1(\Omega))}\le C.
\end{align*}
\end{lemma}

\begin{proof}
We multiply the second equation in \eqref{5.rhoc} by $\pa_t c_\tau$, integrate over $\Omega$ and $(0,t)$, and integrate by parts:
\begin{align*}
  \tau&\int_0^t\int_\Omega(\pa_t c_\tau)^2 dxds
  = -\int_0^t\int_\Omega(\eta\Delta c_\tau - c_\tau + \rho_\tau)
  \pa_t c_\tau dxds \\
  &= -\frac12\int_\Omega\big(\eta|\na c_\tau|^2 + c_\tau\big)dx\Big|_0^t
  + \int_0^t\frac{d}{dt}\int_\Omega\rho_\tau c_\tau dxds
  - \int_0^t\langle\pa_t\rho_\tau,c_\tau\rangle ds \\
  &\le -\frac12\int_\Omega\big(\eta|\na c_\tau(t)|^2 + c_\tau(t)^2\big)dx
  + \int_\Omega\rho_\tau(t)c_\tau(t)dx \\
  &\phantom{xx}+ \|\pa_t\rho_\tau\|_{L^2(0,T;H^1(\Omega)')}
  \|c_\tau\|_{L^2(0,T;H^1(\Omega))} + C(\rho^0,c^0) \\
  &\le -\frac12\int_\Omega\big(\eta|\na c_\tau(t)|^2 
  + c_\tau(t)^2\big)dx + C,
\end{align*}
where the last step follows from Lemma \ref{lem.entD} and $\rho_\tau\le 1$ and $c_\tau\le 1$. This concludes the proof.
\end{proof}

We deduce from Lemma \ref{lem.entD} that $(\rho_\tau^{m/2})$ is bounded in $L^2(0,T;H^1(\Omega))$ and
\begin{align*}
  \|\pa_t \rho_\tau\|_{L^2(0,T;H^1(\Omega)')}
  \le \bigg\|\frac{2}{m}(1-\rho_\tau)\rho_\tau^{m/2}\na\rho_\tau^{m/2}
  - \chi\rho_\tau(1-\rho_\tau)\na c_\tau\bigg\|_{L^2(\Omega_T)}
  \le C
\end{align*}
for some $C>0$ independent of $\tau$. This allows us to apply the Aubin--Lions lemma in the version of \cite{Mou16} to conclude the existence of a subsequence that is not relabeled such that
\begin{align*}
  \rho_\tau\to \rho\quad\mbox{strongly in }L^2(\Omega_T)
  \ \mbox{as }\tau\to 0.
\end{align*}
Because of the uniform $L^\infty(\Omega_T)$ bound for $\rho_\tau$, the family $(\rho_\tau)$ converges strongly in any $L^p(\Omega_T)$ with $p<\infty$. Lemma \ref{lem.entD} shows that $(c_\tau)$ is bounded in $L^2(0,T;H^1(\Omega))$, hence, for a subsequence, $c_\tau\rightharpoonup c$ and $\na c_\tau\rightharpoonup\na c$ weakly in $L^2(\Omega_T)$. By Lemma \ref{lem.ceps}, $\tau\pa_t c_\tau\to 0$ strongly in $L^2(\Omega_T)$. Furthermore, the following weak convergences hold (up to subsequences):
\begin{align*}
  \na\rho_\tau^{m/2}\rightharpoonup\na\rho^{m/2}
  \quad\mbox{weakly in }L^2(\Omega_T), \quad
  \pa_t\rho_\tau\rightharpoonup\pa_t\rho
  \quad\mbox{weakly in }L^2(0,T;H^1(\Omega)').
\end{align*}
This implies that
\begin{align*}
  (1-\rho_\tau)\rho_\tau^{m-1}\na\rho_\tau
  = \frac{2}{m}(1-\rho_\tau)\rho_\tau^{m/2}\na\rho_\tau^{m/2}
  \rightharpoonup \frac{2}{m}(1-\rho)\rho^{m/2}\na\rho^{m/2}
  = (1-\rho)\rho^{m-1}\na\rho
\end{align*}
weakly in $L^1(\Omega_T)$, and since $(1-\rho_\tau)\rho_\tau^{m/2}\na\rho_\tau^{m/2}$ is uniformly bounded in $L^2(\Omega_T)$, this convergence also holds in $L^2(\Omega_T)$. Similarly, 
\begin{align*}
  \rho_\tau(1-\rho_\tau)\na c_\tau\rightharpoonup
  \rho(1-\rho)\na c\quad\mbox{weakly in }L^2(\Omega_T).
\end{align*} 
Thus, we can pass to the limit $\tau\to 0$ in the weak formulation of \eqref{5.rhoc} to infer that $(\rho,c)$ solves equations \eqref{1.tau}, together with the initial and boundary conditions \eqref{1.bic}. 


\subsection{Vanishing diffusion limit}

The aim is the limit $\eta\to 0$ in equations \eqref{1.rho}--\eqref{1.c}. Let $(\rho_\eta,c_\eta)$ be a weak solution to \eqref{1.rho}--\eqref{1.c}. 

\begin{lemma}
Let $1 < m\le 2$. Then there exists $C>0$, independent of $\eta$, such that
\begin{align}\label{6.est}
  \|\rho_\eta\|_{L^2(0,T;H^1(\Omega))}
  + \|c_\eta\|_{L^2(0,T;H^1(\Omega))}
  + \sqrt\eta\|\Delta c_\eta\|_{L^2(\Omega_T)} \le C(T),
\end{align}
\end{lemma}

\begin{proof}
We compute, using Young's inequality,
\begin{align*}
  \frac{\tau}{2}\int_\Omega|\na c_\eta|^2 dx\Big|_0^t
  &= -\int_0^t\int_\Omega(\eta\Delta c_\eta - c_\eps + \rho_\eta)
  \Delta c_\eta dxds \\
  &= -\int_0^t\int_\Omega\big(\eta(\Delta c_\eta)^2 + |\na c_\eta|^2
  \big) dxds 
  + \int_0^t\int_\Omega\na\rho_\eps\cdot\na c_\eps dxds.
\end{align*}
The condition $1 < m\le 2$ implies that $\rho_\eps^{m-2}\ge 1$. Therefore, taking into account \eqref{3.claim} and the previous computation, recalling definition \eqref{1.Feta} of $F_\eta$ and summing the expressions,
\begin{align}\label{5.aux}
  F_\eta(\rho_\eta,c_\eta)\Big|_0^t  
  &\le -\int_0^t\int_\Omega\rho_\eta^{m-2}|\na\rho_\eta|^2 dxds
  - \int_0^t\int_\Omega\big(\eta(\Delta c_\eta)^2 + |\na c_\eta|^2
  \big) dxds \\
  &\phantom{xx} + (\chi+1)\int_0^t\int_\Omega\na\rho_\eta
  \cdot\na c_\eta dxds \nonumber \\
  &\le -\frac12\int_0^t\int_\Omega|\na\rho_\eta|^2 dxds
  - \eta\int_0^t\int_\Omega(\Delta c_\eta)^2 dxds \nonumber \\
  &\phantom{xx}+ \bigg(\frac12(\chi+1)^2-1\bigg)\int_0^t\int_\Omega 
  |\na c_\eta|^2 dxds \nonumber \\
  &\le -\frac12\int_0^t\int_\Omega|\na\rho_\eta|^2 dxds
  - \eta\int_0^t\int_\Omega(\Delta c_\eta)^2 dxds \nonumber \\
  &\phantom{xx}+ \bigg(\frac12(\chi+1)^2-1\bigg)\int_0^t (F_\eta(\rho_\eta,c_\eta)+2)ds. \nonumber 
\end{align}
By Gronwall's inequality,
\begin{align*}
  F_\eta(\rho_\eta(t),c_\eta(t)) \le C(T)F_\eta(\rho^0,c^0)
   \quad\mbox{for }t\le T.
\end{align*}
Going back to \eqref{5.aux}, we obtain \eqref{6.est}, finishing the proof.
\end{proof}

The previous lemma shows that
\begin{align*}
  \|\pa_t\rho_\eta\|_{L^2(0,T;H^1(\Omega)')}
  &\le \|(1-\rho_\eta)\rho_\eta^{m-1}\na\rho_\eta
  - \chi\rho_\eta(1-\rho_\eta)\na c_\eta\|_{L^2(\Omega_T)}\le C, \\
  \|\pa_t c_\eta\|_{L^2(\Omega_T)}
  &\le \|\eta\Delta c_\eta - c_\eta + \rho_\eta\|_{L^2(\Omega_T)}
  \le C,
\end{align*}
and we can apply the classical Aubin--Lions lemma to find that, for a subsequence,
\begin{align*}
  \rho_\eta\to \rho, \quad c_\eta\to c 
  \quad\mbox{strongly in }L^2(\Omega_T)\ \mbox{as }\eta\to 0,
\end{align*}
and the convergences also hold in $L^p(\Omega_T)$ with $p<\infty$. Moreover, again for subsequences,
\begin{align*}
  & \na\rho_\eta\rightharpoonup\na\rho, \quad
  \na c_\eta\rightharpoonup\na c, \quad
  \pa_t c_\eta\rightharpoonup\pa_t c 
  \quad\mbox{weakly in }L^2(\Omega_T), \\
  & \pa_t\rho_\eta\rightharpoonup\na\rho 
  \quad\mbox{weakly in }L^2(0,T;H^1(\Omega)'), \quad
  \eta\Delta c_\eta\to 0 \quad\mbox{strongly in }L^2(\Omega_T).
\end{align*}
These convergences are sufficient to pass to the limit $\eta\to 0$ in the weak formulation of \eqref{1.rho}--\eqref{1.bic} to infer that $(\rho,c)$ solves \eqref{1.eta}, together with the initial and boundary conditions \eqref{1.bic}. 


\section{Numerical simulations}\label{sec.num}

We present numerical experiments using a one-dimensional finite-volume scheme. 

\subsection{Numerical scheme}

Equations \eqref{1.rho}--\eqref{1.c} are solved in one space dimension by an upwind finite-volume scheme. The domain $\Omega=(0,1)$ is divided into the uniform intervals $I_i=[x_{x-1/2},x_{i+1/2}]$, where $x_{i\pm 1/2}=(i\pm 1/2)\Delta x$, $\Delta x>0$, and $i=1,\ldots,N$. The time steps are given by $t_k=k\Delta t$, where $\Delta t>0$ and $k\in\N$. The initial densities are defined by
\begin{align*}
  \rho_i^0 = \frac{1}{\Delta x}\int_{I_i}\rho^0 dx, \quad
  c_i^0 = \frac{1}{\Delta x}\int_{I_i}c^0 dx, \quad i=1,\ldots,N.
\end{align*}
The scheme reads as
\begin{align*}
  \frac{1}{\Delta t}(\rho_i^{k} - \rho_i^{k-1})
  + \frac{1}{\Delta x}(F_{i+1/2}^k - F_{i-1/2}^k) &= 0, \\
  \frac{1}{\Delta t}(c_i^{k} - c_i^{k-1})
  - \frac{1}{\Delta x^2}(c_{i+1}^{k-1} - 2c_i^{k-1} + c_{i-1}^{k-1})
  + c_i^{k-1} - \rho_i^{k-1} &= 0, 
\end{align*}
where $i=1,\ldots,N$, the fluxes are defined by
\begin{align*}
  F_{i+1/2}^k = -\frac{1}{\Delta x}\big(\psi_{i+1/2}^k
  (\rho_{i+1}^k-\rho_i^k) - \chi\phi_{i+1/2}^k(c_{i+1}^k-c_i^k)\big),
  \quad i=1,\ldots,N-1,
\end{align*}
$F_{1/2}^k=F_{N+1/2}^k=0$, and the interface values $\psi_{i+1/2}^k$ and $\phi_{i+1/2}^k$ are given by the upwind scheme
\begin{align*}
  \psi_{i+1/2}^k &= \begin{cases}
  (1-\rho_i^k)(\rho_{i+1}^k)^{m-1} 
  &\mbox{if }\rho_{i+1}^k-\rho_i^k\ge 0, \\
  (1-\rho_{i+1}^k)(\rho_{i}^k)^{m-1} 
  &\mbox{if }\rho_{i+1}^k-\rho_i^k < 0,
  \end{cases} \\
  \phi_{i+1/2}^k &= \begin{cases}
  \rho_i^k(1-\rho_{i+1}^k) &\mbox{if }c_{i+1}^k-c_i^k\ge 0, \\
  \rho_{i+1}^k(1-\rho_{i}^k) &\mbox{if }c_{i+1}^k-c_i^k < 0.
  \end{cases}
\end{align*}
We choose the numerical parameters $\Delta x=0.01$ and $\Delta t=10^{-6}$. The nonlinear discrete system is solved by Newton's method.


In Figure \ref{fig.chi}, we illustrate the long-time behavior of $(\rho,c)$ for $\chi=1$ (top row) and $\chi=10$ (bottom row). The parameters are $m=2$, $\tau=1$, and $M=1/2$. In accordance with the theoretical results of Proposition \ref{prop.unique}, the solution $(\rho,c)(t)$ converges to the unique constant steady state $(M,M)$ if $\chi=1$, while $(\rho,c)(t)$ converges to a nonconstant steady state if $\chi$ is sufficiently large (here $\chi=10$). 

\begin{figure}[ht]
\includegraphics[width=40mm]{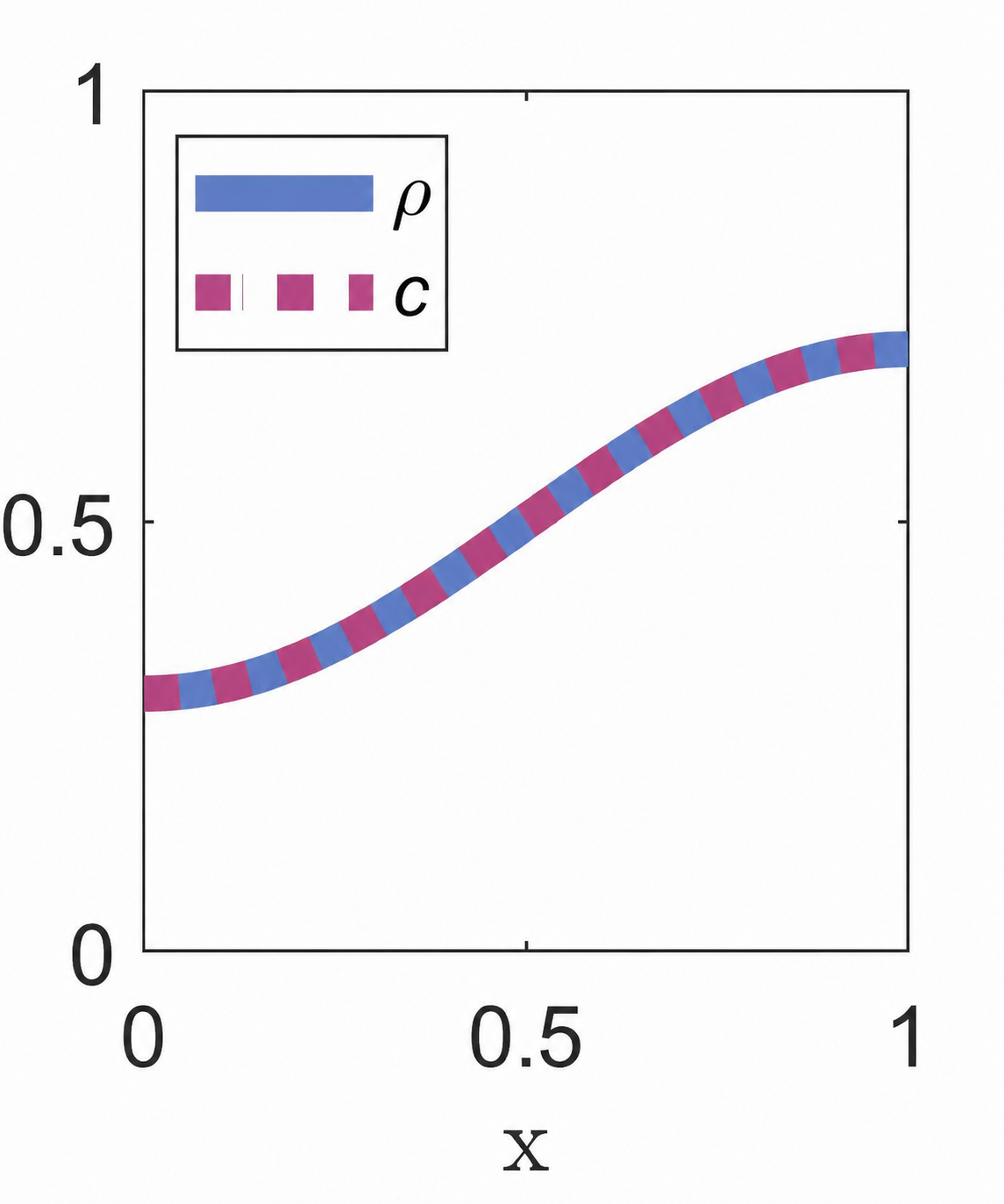}
\includegraphics[width=40mm]{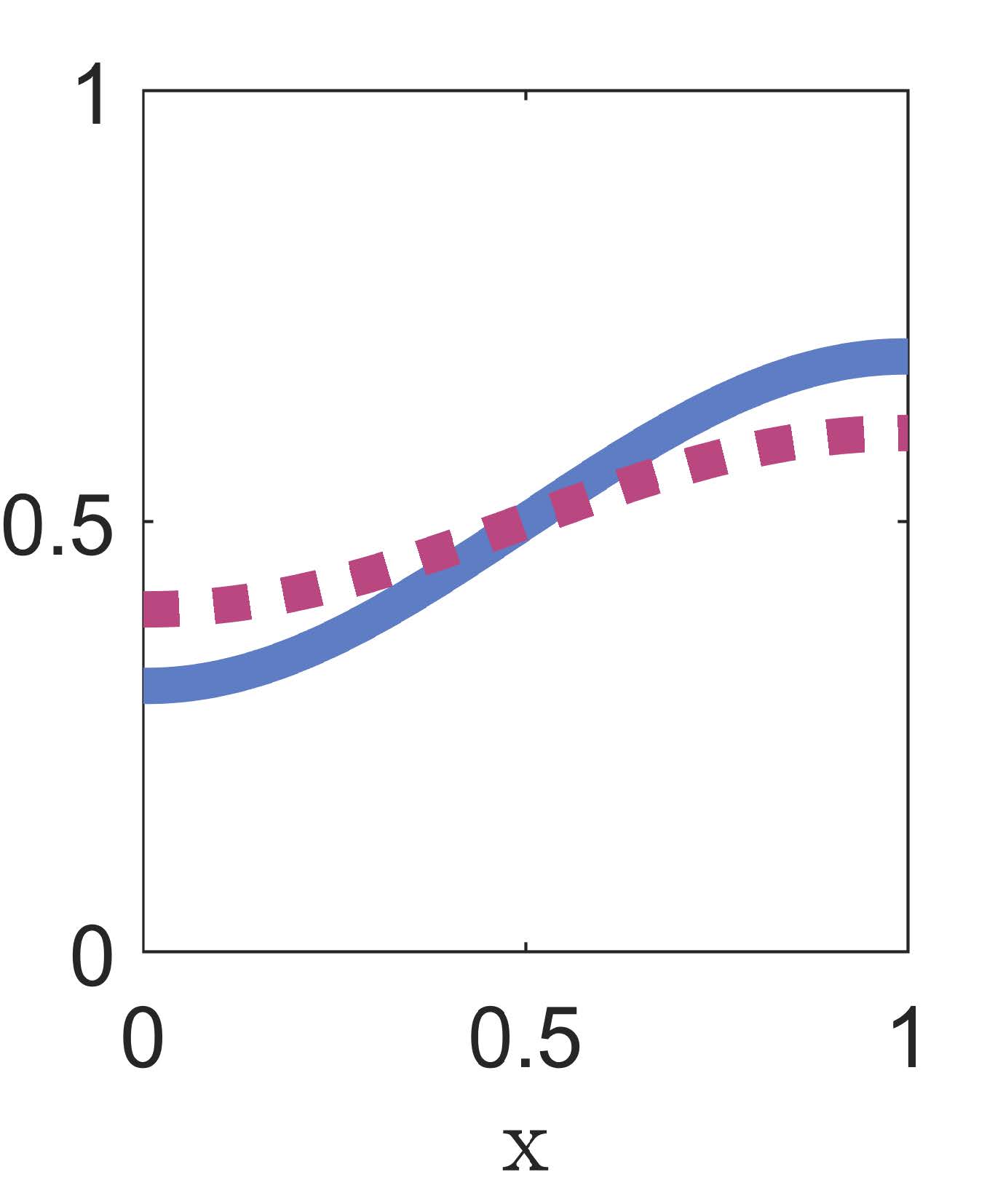}
\includegraphics[width=40mm]{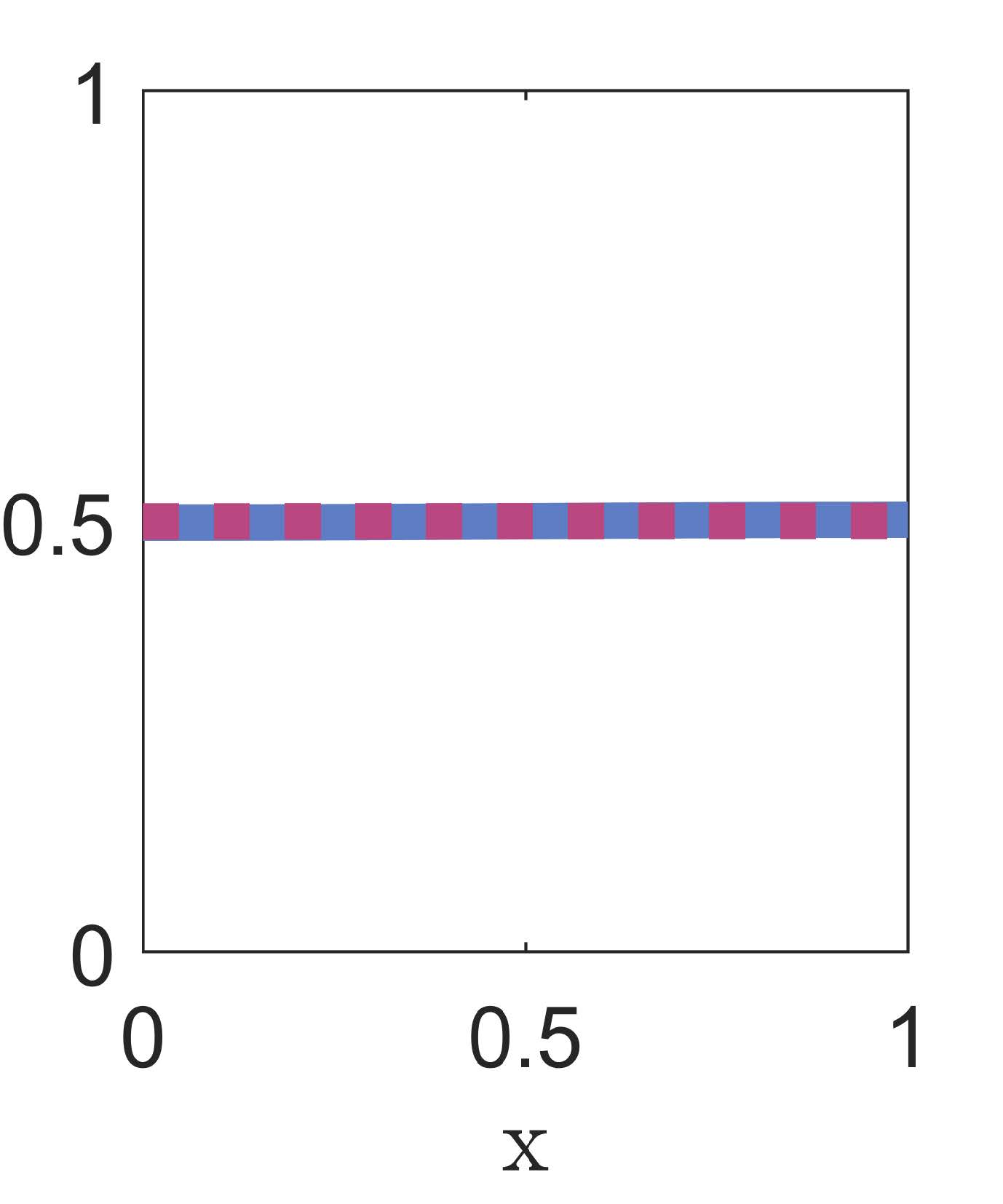}
\includegraphics[width=40mm]{fig_chi1_t0.png}
\includegraphics[width=40mm]{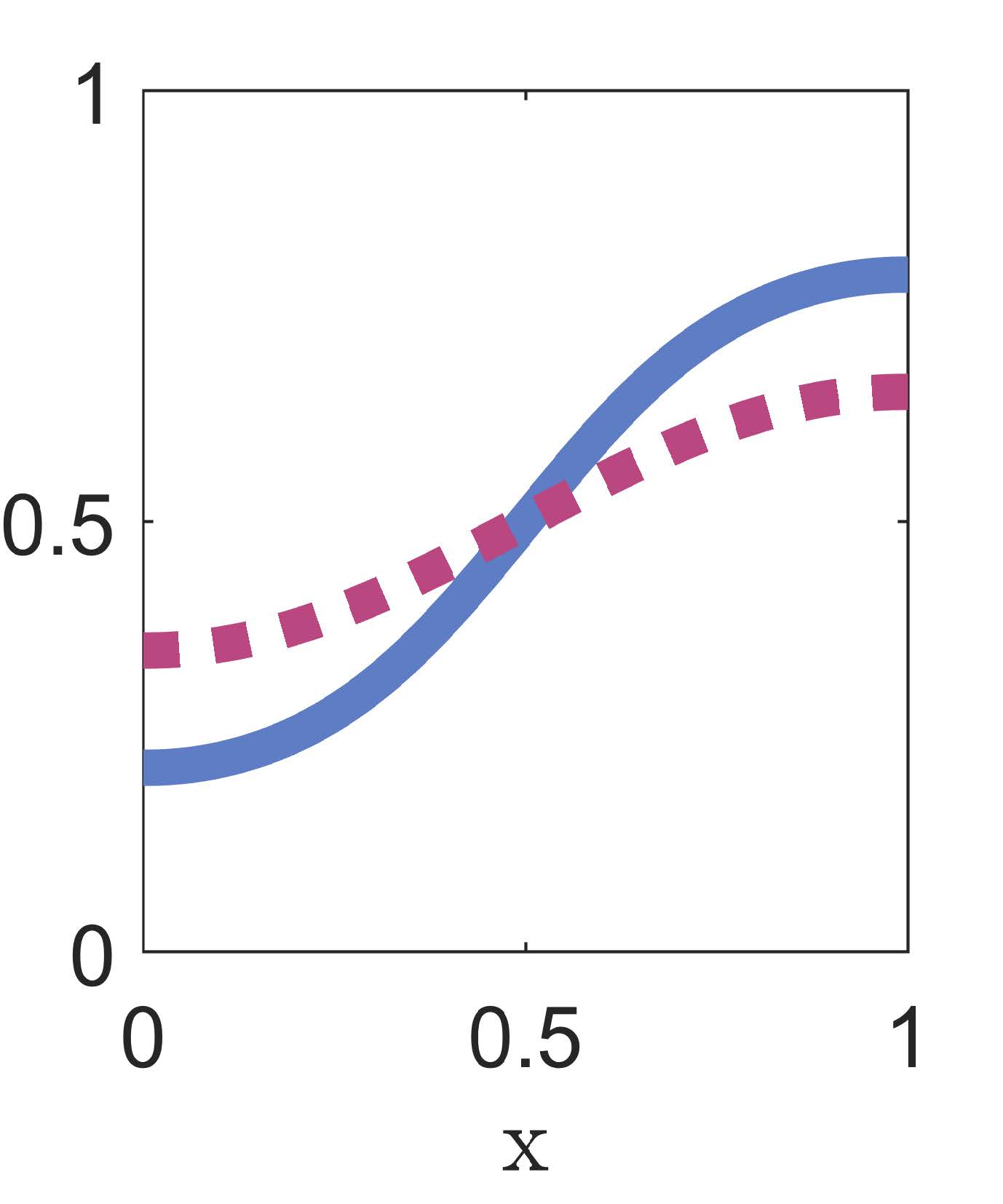}
\includegraphics[width=40mm]{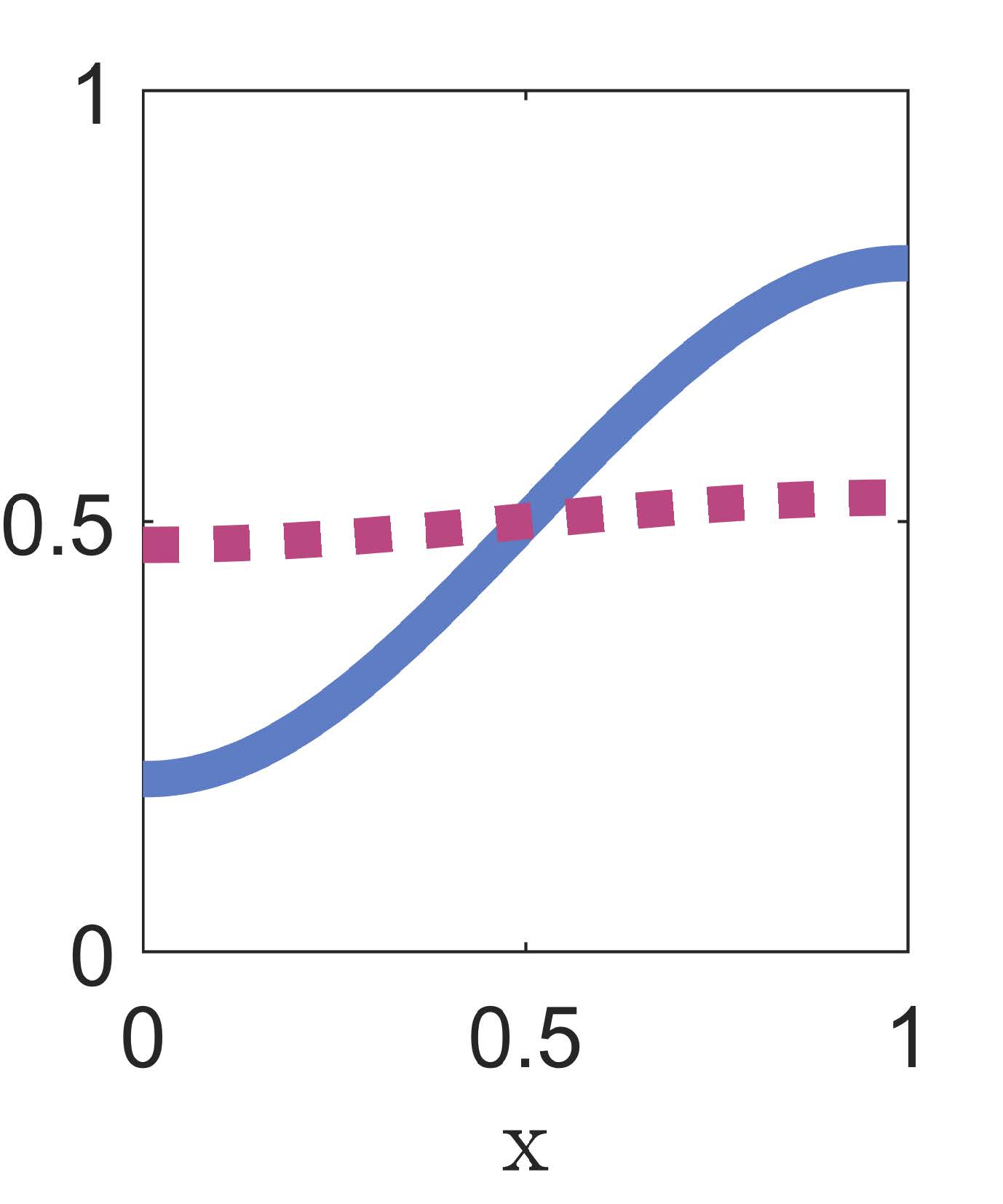}
\caption{Time evolution of $\rho$ (solid line) and $c$ (dotted line) with $\chi=1$ (top row) and $\chi=10$ (bottom row) at times $t=0$ (left), $t=5$ (middle), and $t=100$ (right).}
\label{fig.chi}
\end{figure}


Next, we present snapshots of the solution $(\rho,c)$ for different values of $\tau$, choosing $t=10$, $m=2$, and $M=1/2$. Figure \ref{fig.D} shows that the solution to the parabolic--parabolic system with $\tau=1$ clearly differs from the one of the parabolic--elliptic system with $\tau=0$. The difference between both solutions become smaller when $\tau\to 0$, and they coincide when $\tau=0$. 

\begin{figure}[ht]
\includegraphics[width=40mm]{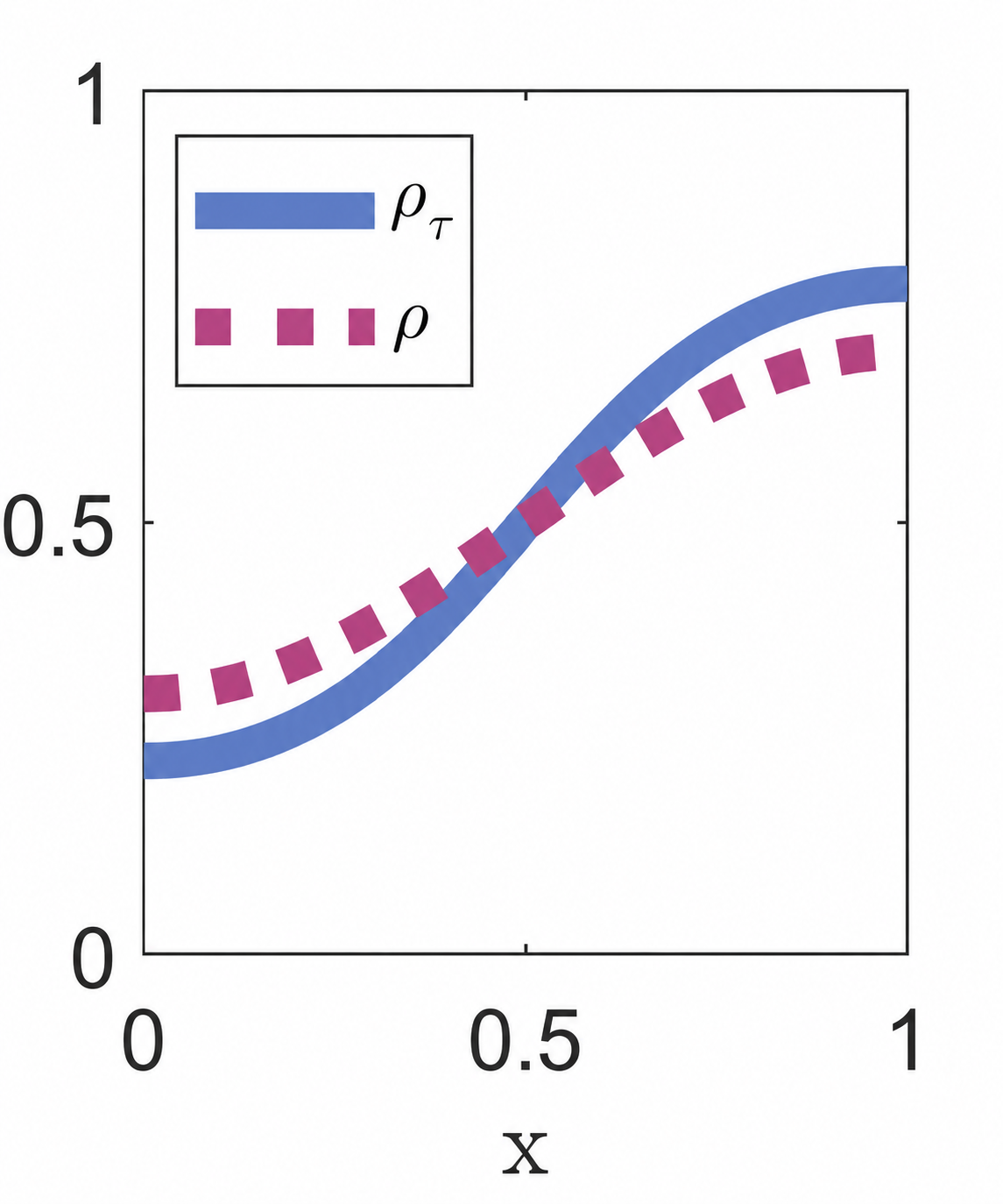}
\includegraphics[width=40mm]{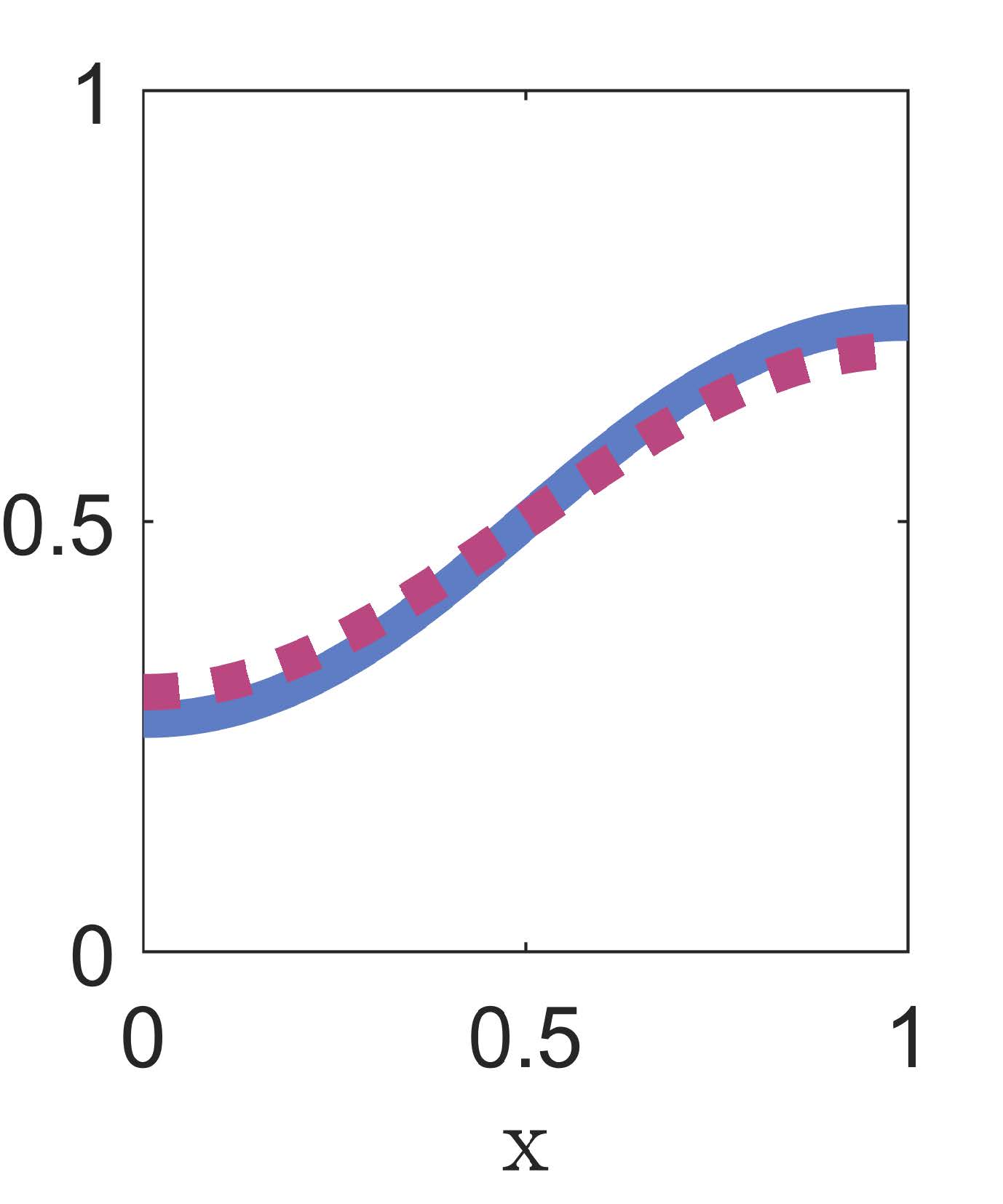}
\includegraphics[width=40mm]{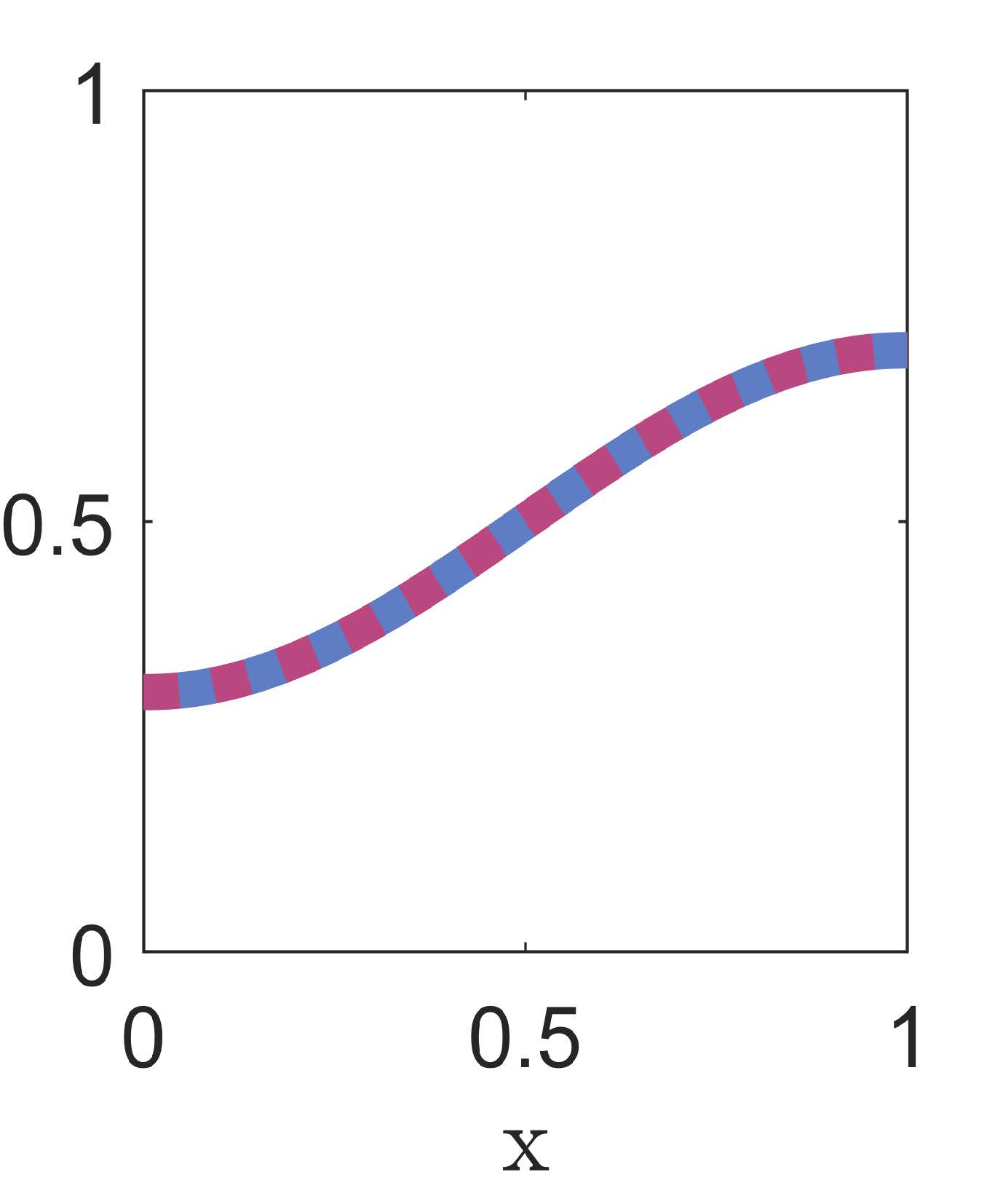}
\caption{Time snapshot of $\rho_\tau$ (solid line) and $\rho$ (dotted line) with $\tau=\eps=1$ (left), $\tau=0.1$ (middle), and $\tau=0$ (right).}
\label{fig.D}
\end{figure}


Finally, we study the behavior of the solution $(\rho,c)$ for different values of $\eta$. We choose the parameters $t=10$, $m=2$, $\tau=1$, and $M=1/2$. Figure \ref{fig.eta} illustrates the numerical convergence of the solution $\rho_\eta$ to the solution $\rho$ to the reduced system \eqref{1.eta} as $\eta\to 0$. The solutions $\rho_\eta$ and $\rho$ converge towards each other as $\eta$ decreases.

\begin{figure}[ht]
\includegraphics[width=40mm]{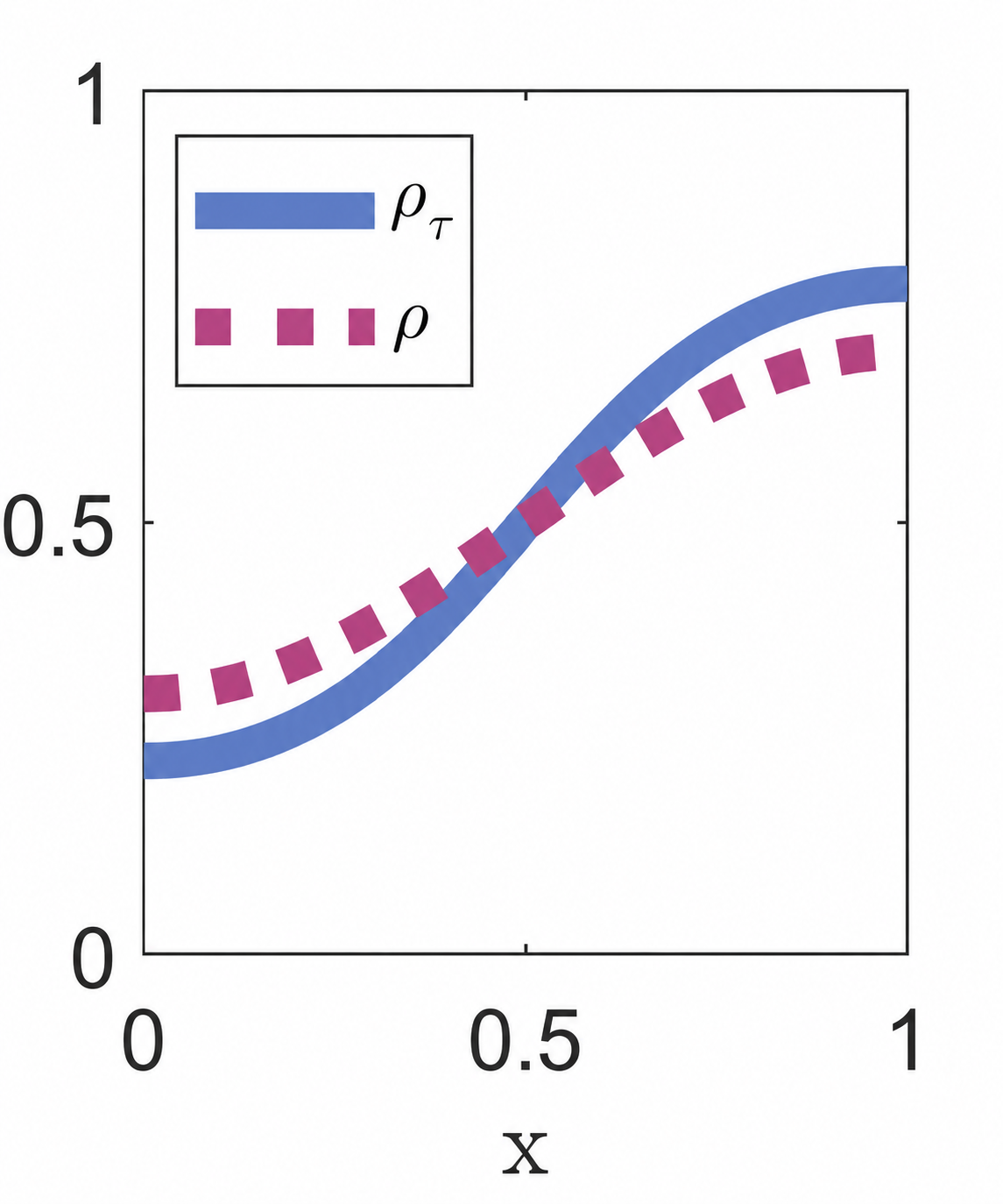}
\includegraphics[width=40mm]{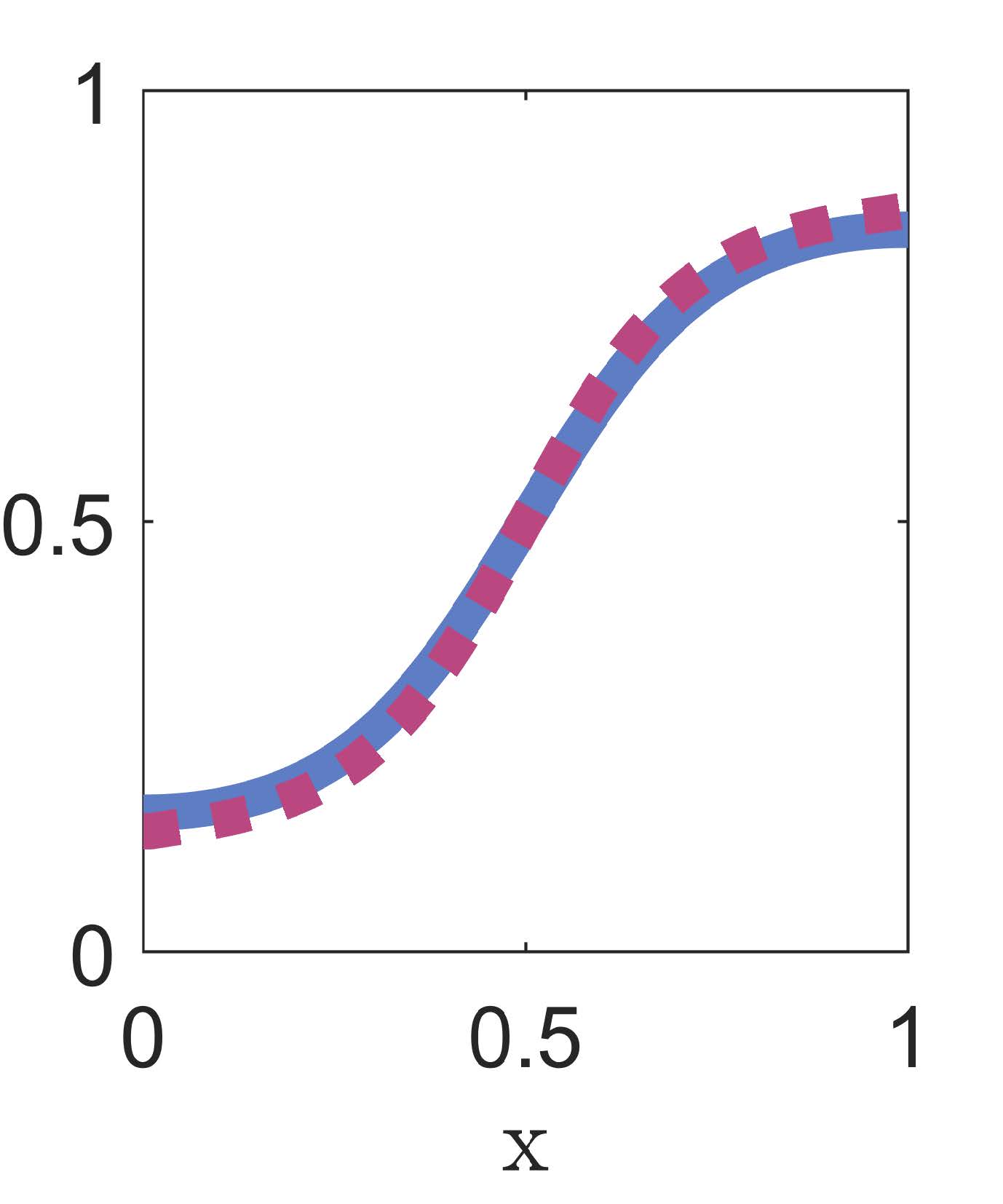}
\includegraphics[width=40mm]{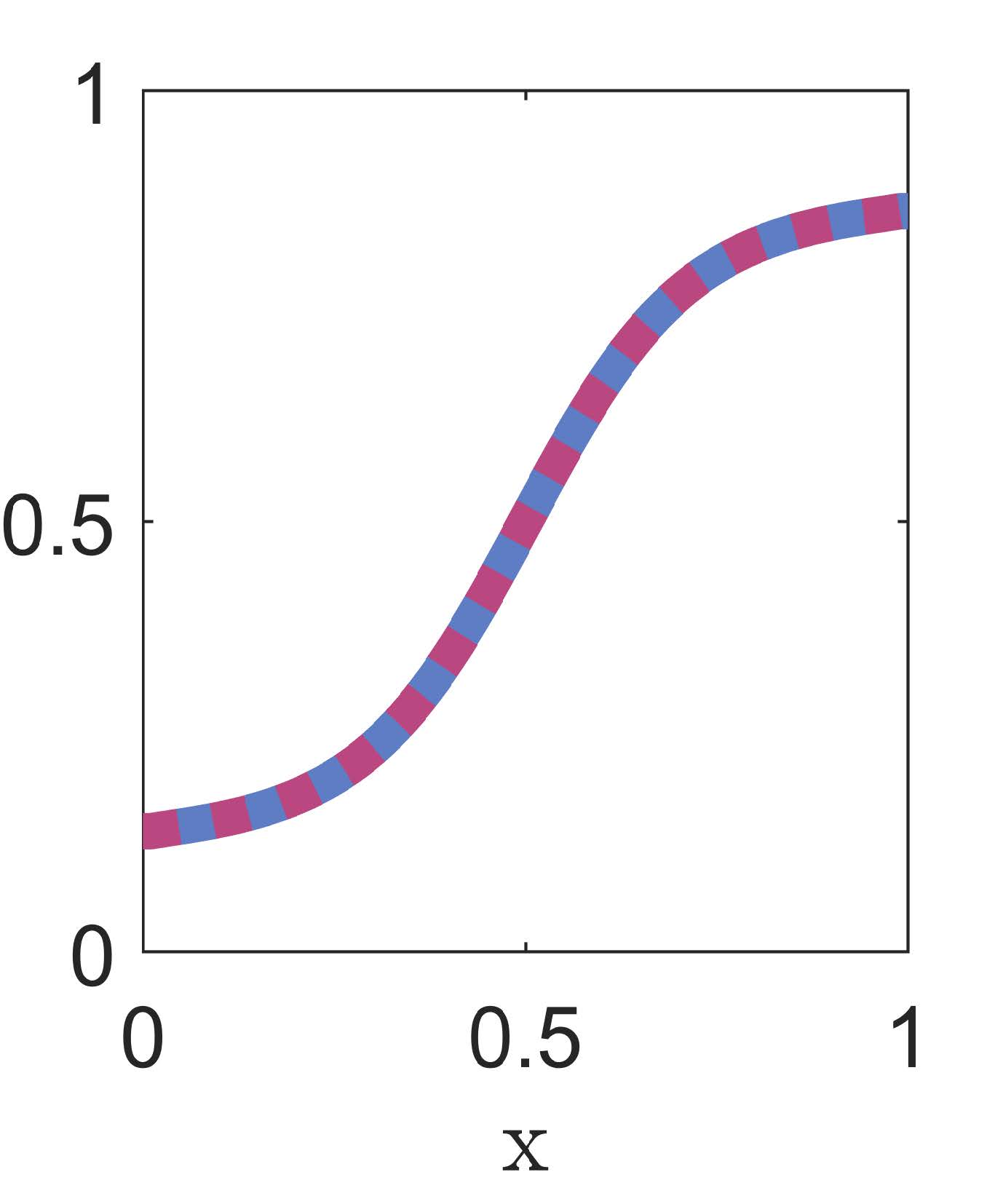}
\caption{Time snapshot of $\rho_\eta$ (solid line) and $\rho$ (dotted line) with $\eta=5$ (left), $\eta=0.5$ (middle), and $\eta=0$ (right).}
\label{fig.eta}
\end{figure}


\begin{appendix}
\section{Derivation of the chemotaxis model}\label{app}

We derive equation \eqref{1.rho} for the cell density using a multiphase approach inspired by \cite{LKBJS06}. Let $\rho_0(x,t)$ and $\rho_1(x,t)$ be the volume fractions occupied by water and cells, respectively, and let $v_0(x,t)$ and $v_1(x,t)$ be the corresponding velocities. The starting point are the mass and momentum balance equations
\begin{align}\label{a.eq}
  \pa_t\rho_i + \diver(\rho_i v_i) = 0, \quad 
  \na(\rho_i p_i) = F_i \quad\mbox{for }i=0,1,
\end{align}
where $p_i$ is the phase-specific pressure and $F_i$ represents the force acting between the phases. We assume that the total flux vanishes, $\rho_0v_0+\rho_1v_1=0$. Then the sum of the mass balance equations yields $\pa_t(\rho_0+\rho_1)=0$, and assuming that $\rho_0+\rho_1=1$ initially, we conclude that $(\rho_0+\rho_1)(t)=1$ for all times $t>0$. Thus, it is sufficient to derive an equation for the cell fraction $\rho_1$.  

The pressures are assumed to be defined by
\begin{align*}
  p_1 = p + q, \quad p_0 = p, \quad\mbox{where }
  q = \frac{1}{m}\rho_1^{m-1}.
\end{align*}
Here, $p$ is the overall pressure that is common to both the cells and water and $q$ is the intraphase pressure that is only present in the cell phase. The interaction force $F_i = F_{i,{\rm ip}} + F_{i,{\rm ch}} + F_{i,{\rm vd}}$ is supposed to be the sum of the interphase, chemotactic, and viscous drag forces. Following \cite{LKBJS06}, the interphase forces read as 
\begin{align*}
  F_{1,{\rm ip}} = p(\rho_0\na\rho_1 - \rho_1\na\rho_0) = p\na\rho_1,
  \quad
  F_{0,{\rm ip}} = p(\rho_1\na\rho_0 - \rho_0\na\rho_1) = -p\na\rho_1.
\end{align*}
The chemotactic force $F_{i,{\rm ch}}$ is the difference between the force $f_i$ (action) and the opposing force $f'_i$ (reaction). The action $f_1$ is given by the gradient of the chemical signal, weighted by the volume fraction $\rho_1$, giving $f_1 = \chi\rho_1\na c$ for some constant $\chi>0$. The reaction $f'_1$ is shared among all the phases present in the fluid, again weighted by the volume fraction, $f'_1 = \rho_1(f_0+f_1)$. We suppose that the force on the water phase vanishes, $f_0=0$. Consequently, we obtain
\begin{align*}
  F_{1,{\rm ch}} &= f_1 - \rho_1(f_0+f_1) = \chi\rho_1(1-\rho_1)\na c, \\
  F_{0,{\rm ch}} &= f_0 - \rho_0(f_0+f_1) = -\chi\rho_1(1-\rho_1)\na c.
\end{align*}
Finally, the drag forces are given by 
\begin{align*}
  F_{1,{\rm vd}} = -k\rho_0\rho_1(v_0-v_1) = -k\rho_1v_1, \quad
  F_{0,{\rm vd}} = -k\rho_0\rho_1(v_1-v_0) = -k\rho_0v_0,
\end{align*}
where $k>0$ is the drag coefficient and we used $\rho_0v_0=-\rho_1v_1$. This implies that
\begin{align*}
  F_1 = p\na\rho_1 - \chi\rho_1(1-\rho_1)\na c - k\rho_1 v_1, \quad
  F_0 = p\na\rho_0 + \chi\rho_1(1-\rho_1)\na c - k\rho_0 v_0.
\end{align*}
Observe that the total force vanishes, $F_0+F_1=0$, which means that there are no external forces.

We insert the expression for the forces into the momentum balance equation in \eqref{a.eq}:
\begin{align}\label{a.aux}
  0 &= -\na(\rho_1p_1) + F_1
  = -\na\bigg(\rho_1 p + \frac{1}{m}\rho_1^m\bigg)
  + p\na\rho_1 - \chi\rho_1(1-\rho_1)\na c - k\rho_1 v_1 \\
  &= -\rho_1\na p - \rho_1^{m-1}\na\rho_1 - \chi\rho_1(1-\rho_1)\na c 
  - k\rho_1 v_1, \nonumber \\
  0 &= -\na(\rho_0p_0) + F_0
  = -\na(\rho_0 p) + p\na\rho_0 + \chi\rho_1(1-\rho_1)\na c 
  - k\rho_0 v_0 \nonumber \\
  &= -\rho_0\na p + \chi\rho_1(1-\rho_1)\na c - k\rho_0 v_0. \nonumber 
\end{align}
The sum of both equations leads to an expression for $\na p$ (since $\rho_0+\rho_1=1$):
\begin{align*}
  \na p = -\rho_1^{m-1}\na\rho_1.
\end{align*}
We insert this expression into \eqref{a.aux}:
\begin{align*}
  0 &= -\rho_1\na p - \rho_1^{m-1}\na\rho_1 - \chi\rho_1(1-\rho_1)\na c 
  - k\rho_1 v_1 \\
  &= (1-\rho_1)\rho_1^{m-1}\na\rho_1 - \chi\rho_1(1-\rho_1)\na c 
  - k\rho_1 v_1.
\end{align*}
Finally, we solve this equation for $\rho_1v_1$ and insert the terms into \eqref{a.eq}:
\begin{align*}
  \pa_t\rho_1 &= k^{-1}\diver\big((1-\rho_1)\rho_1^{m-1}\na\rho_1
  - \chi\rho_1(1-\rho_1)\na c\big) \\
  &= k^{-1}\diver\big[\rho_1(1-\rho_1)
  \na\big(\rho_1^{m-1}/(m-1)-\chi c\big)\big],
\end{align*}
which equals \eqref{1.rho} with $k=1$. 

We remark that a similar derivation has been performed in \cite[Sec.~2]{ByOw04}. In that work, the chemotaxis effect is taken into account in the pressure $p_1$ in the cell phase instead in the chemotaxis force $F_{1,{\rm ch}}$ by setting $p_1 = p + q(\rho_1,c)$, which leads to
\begin{align*}
  \pa_t\rho_1 = \diver\big((1-\rho_1)\na(\rho_1 q(\rho_1,c))\big)
  = \diver\big(\rho_1(1-\rho_1)\na q(\rho_1,c)
  + (1-\rho_1)q(\rho_1,c)\na c\big).
\end{align*} 
This model is different from ours because of the additional gradient of the chemical concentration.
\end{appendix}


\end{document}